\documentclass[9pt, reqno]{amsart}
\usepackage{graphicx, amssymb, amsmath, amsthm}
\numberwithin{equation}{section}

\usepackage{hyperref}

\newcommand{\R}{\mathbb{R}}

\newtheorem{theorem}{Theorem}[section]

\newtheorem{lemma}[theorem]{Lemma}

\newtheorem{proposition}[theorem]{Proposition}

\theoremstyle{definition}

\newtheorem{remark}[theorem]{Remark}

\makeatletter
\newcommand{\Extend}[5]{\ext@arrow0099{\arrowfill@#1#2#3}{#4}{#5}}
\makeatother

\begin{document}
\title[Dispersive and Strichartz estimates]{Dispersive and Strichartz estimates for 3D wave equation with a class of many-electric potentials}

\author{Haoran Wang}
\address{Department of Mathematics, Beijing Institute of Technology, Beijing 100081, China;}
\email{wanghaoran@bit.edu.cn}

\begin{abstract}
We prove the dispersive and Strichartz estimates for solutions to the wave equation with a class of many-electric potentials in spatial dimension three. To obtain the desired dispersive estimate, based on the spectral properties of the Schr\"odinger operator involved, we subsequently prove the dispersive estimate for the corresponding Schr\"odinger semigroup, obtain a Gaussian type upper bound, establish Bernstein type inequalities, and finally pass to the M\"uller-Seeger's subordination formula. The desired Strichartz estimates follow by the established dispersive estimate and the standard argument of Keel-Tao.
\end{abstract}

\maketitle

\begin{center}
 \begin{minipage}{100mm}
 \end{minipage}
 \end{center}
\bigskip\bigskip
\maketitle

\section{Introduction}

In this paper, we consider the Cauchy problem for the 3D perturbed wave equation
\begin{equation}\label{eq:wave}
\begin{cases}
\partial_{tt}u-\Delta u+Vu=0,\quad (t,x)\in\R\times\R^3\\
u(0,x)=f(x),\quad \partial_t u(0, x)=g(x),
\end{cases}
\end{equation}
where $V(x):\R^3\rightarrow\R$ is the Laguerre potential, a combination of the inverse square potential and the Hermite potential, i.e.
\begin{equation}\label{potential}
V(x)=\frac{a}{|x|^2}+\frac{b}{4}|x|^2,\quad a,b\geq0,\quad x\in\R^3\setminus\{0\}.
\end{equation}
If we write $H_{a,b}$ for the operator in \eqref{eq:wave}, i.e.
\begin{equation}\label{H-ab}
H_{a,b}:=-\Delta+\frac{a}{|x|^2}+\frac{b}{4}|x|^2,
\end{equation}
then it is easy to verify that the quadratic form of the operator $H_{a,b}$ is positive definite, i.e.
\begin{equation*}
\int_{\R^3}\left(|\nabla\varphi|^2+\frac{a}{|x|^2}|\varphi|^2+\frac{b|x|^2}{4}|\varphi|^2\right)\mathrm{d}x>0,\quad \forall a,b\geq0,\quad \varphi\in\mathcal{S}(\R^3).
\end{equation*}
Hence, the operator $H_{a,b}$ is a symmetric semi-bounded operator on $L^2(\R^3;\mathbb{C})$, which implies that $H_{a,b}$ admits a self-adjoint extension, particularly, the Friedrichs' extension, with a formal domain $\mathcal{D}$ given by
\begin{equation*}
\mathcal{D}(H_{a,b})=\left\{\varphi\in L^2(\R^3;\mathbb{C}):H_{a,b}\varphi\in L^2(\R^3;\mathbb{C})\right\}.
\end{equation*}
As a consequence of the Friedrichs' extension, the unitary propagators like $e^{itH_{a,b}},e^{it\sqrt{H_{a,b}}}$ are well-defined on the domain of $H_{a,b}$ via the spectral theorem.

\noindent The unique solution $u\in C(\R;L^2(\R^3))$ of the problem \eqref{eq:wave} can be written as
\begin{equation}\label{wave:solution}
u(t,x)=\cos(t\sqrt{H_{a,b}})f(x)+\frac{\sin(t\sqrt{H_{a,b}})}{\sqrt{H_{a,b}}}g(x).
\end{equation}
If the inverse square potential is removed (i.e. $a\equiv0$ in \eqref{H-ab}), then the Hamiltonian $H_{a,b}$ is reduced to the usual harmonic oscillator $H_{0,b}=-\Delta+\frac{b}{4}|x|^2$. If the Hermite potential is removed (i.e. $b=0$ in \eqref{H-ab}), then one gets the Schr\"odinger operator with an inverse square potential $H_{a,0}=-\Delta+\frac{a}{|x|^2}$, which is known to have the same homogeneity as the free Laplacian $-\Delta$. If the potential $V$ in \eqref{eq:wave} vanishes, then we recover the classical free wave equation, which is known to be scaling-invariant.
The scaling-invariance of dispersive equations (including the free Schr\"odinger and wave equations) is critical for a large class of dispersive estimates, such as time-decay, Strichartz and local-smoothing.

\noindent In particular, one has the classical dispersive estimate
\begin{equation}\label{dis:free}
\left\|\frac{\sin(t\sqrt{-\Delta})}{\sqrt{-\Delta}}g\right\|_{L^\infty(\R^3)}\leq\frac{C}{|t|}\|g\|_{\dot{B}^1_{1,1}(\R^3)},\quad \forall |t|>0,
\end{equation}
where $\|\cdot\|_{\dot{B}^1_{1,1}}$ denotes the standard homogeneous Besov norm and $C>0$ is a suitable constant independent of $t$ and $g$.

In the past decades, it turns out that such estimates as \eqref{dis:free} play a fundamental role in various mathematical fields, including scattering theory, harmonic analysis and nonlinear dynamics (see e.g. \cite{BCT04,Bec14,KS07,Szp01}). By the standard $TT^\ast$ argument together with the unitary property of the half-wave propagator $e^{it\sqrt{-\Delta}}$
\begin{equation*}
\|e^{it\sqrt{-\Delta}}f\|_{L^2(\R^3)}=\|f\|_{L^2(\R^3)},
\end{equation*}
one can derive the Strichartz estimates
\begin{equation}\label{str:free}
\left\|\cos(t\sqrt{-\Delta})f+\frac{\sin(t\sqrt{-\Delta})}{\sqrt{-\Delta}}g\right\|_{L_t^q(\R;L^r(\R^3))}\leq C\left(\|f\|_{\dot{H}^s(\R^3)}+\|g\|_{\dot{H}^{s-1}(\R^3)}\right)
\end{equation}
for some constant $C>0$, where the pair $(q,r)$ satisfies
\begin{equation}\label{adm:scaling}
\frac{1}{q}+\frac{1}{r}\leq\frac{1}{2},\quad2\leq q,r<\infty
\end{equation}
and $s$ is the gap index
\begin{equation}\label{cd:gap}
s=3\left(\frac{1}{2}-\frac{1}{r}\right)-\frac{1}{q}.
\end{equation}
We assume $0\leq s<\frac{3}{2}$ to guarantee a room for the pair $(q,r)$.
The Strichartz estimate \eqref{str:free} was initially obtained by Segal \cite{Seg76} and then generalized by Strichartz \cite{Str77} in connection with Tomas's restriction theorem \cite{Tom75}. Later, Ginibre and Velo \cite{GV95} introduced a new viewpoint, which was extensively applied by Yajima \cite{Yaj87} to obtain a large class of inequalities for the Schr\"odinger equation. Finally, Keel and Tao \cite{KT98} settled the most challenging endpoint case $q=2$, via bilinear techniques, for an abstract propagator fulfilling the dispersive estimate as \eqref{dis:free}.

We stress that the above mentioned scaling-invariance is no longer valid for our model \eqref{eq:wave} due to the appearance of the Hermite potential in \eqref{H-ab}. Therefore, it is not trivial to verify the dispersive estimate \eqref{dis:free} (and hence the Strichartz estimates \eqref{str:free}) for the wave equation \eqref{eq:wave}, which is the main purpose of this work.
Before stating the main theorems, let us briefly review some relevant papers to better frame our results.
The dispersive estimate \eqref{dis:free} has been obtained previously for the Schr\"odinger and wave equations in e.g. \cite{BS93,Beal94,Bec11,BG12,BDH16,DP05,GV03,GS04,GVV06,Gol06,Gold06,Gold12,JSS91,PST03,RS04} when the potential $V$ is subcritical (i.e. decaying faster than the inverse square potential $|x|^{-2}$ near infinity). In the earlier work \cite{BS93}, Beals and Strauss obtained the $L^p-L^q$ decay estimates for the solution to the wave equation \eqref{eq:wave} but not the dispersive estimate in the spirit of \eqref{dis:free} for potentials both smooth and small in a suitable sense. Later, Georgiev and Visciglia \cite{GV03} proved the dispersive estimate \eqref{dis:free} for almost critical potentials, more precisely, for potentials $V\in C^\delta(\R^3\setminus\{0\}),\delta\in(0,1)$ satisfying
\begin{equation*}
0\leq V(x)\leq\frac{C}{|x|^{2+\epsilon}+|x|^{2-\epsilon}}\quad \text{for some}\quad \epsilon>0.
\end{equation*}
For the critical inverse square potential $\frac{a}{|x|^2}$, it turns out that the dispersive estimate \eqref{dis:free} is true when $a\geq0$ but may be false when $a<0$ (see \cite{PST03,PSTZ03}). The Strichartz estimate \eqref{str:free} was proved by Burq, Planchon, Stalker, and Tahvildar-Zadeh \cite{BPST03,BPST04} for both Schr\"odinger and wave equations with the potential $\frac{a}{|x|^2}$ in two and higher dimensions; actually, they obtained suitable Morawetz-type estimates for the perturbed resolvent by the related multiplier results, which, together with the well-known free counterparts, yields the desired Strichartz estimates for the perturbed operator $-\Delta+V$. Mizutani \cite{Miz17} studied the Strichartz estimates for the Schr\"odinger equation with the critical inverse-square potential $-\frac{(d-2)^2}{4|x|^2}$ in dimension $d\geq3$. It turns out that the inverse-square potential represents a threshold for the validity of the dispersive and Strichartz estimates, as shown in \cite{GVV06}.
For the Schr\"odinger flow $e^{it(\Delta-V)}$, a typical perturbation argument consists of expressing the action of the flow $e^{it(\Delta-V)}$ by the spectral theorem and then reducing the matters of proving the desired dispersive estimate to perform a suitable analysis on the resolvent of the perturbed Laplacian $-\Delta+V$, in the sense of Agmon-H\"ormander. As a standard example of such a method, we refer the interested reader to the paper of Goldberg and Schlag \cite{GS04} (see also Rodnianski and Schlag \cite{RS04}), in which the dispersive estimate was obtained for the Schr\"odinger equation with integrability or decaying conditions imposed on the potential $V$. Another effective method is studying the mapping properties of the wave operators of the perturbed Laplacian on $L^p$ and then the dispersive estimate for the Schr\"odinger flow $e^{it(\Delta-V)}$ follows from the dispersive estimate for the free Laplacian \eqref{dis:free} and the intertwining property. The wave operator argument was introduced and applied by Yajima in his series of papers \cite{Yaj93,Yaj95,Yaj99,Yaj06} to obtain a large class of inequalities for dispersive equations. It turns out that this approach leads to a much stronger result since many inequalities including the resolvent estimate follow as a consequence of the $L^p$-boundedness of the associated wave operators (see e.g. \cite{Bec14,DF06,MSZ23,Wed99}).
For growing potentials, we only mention the interesting paper \cite{DPR10}, in which the dispersive estimates for the wave equation associated to the Hermite operator and the twisted Laplacian were obtained.

In light of the above considerations, the main purpose of this paper is to prove the dispersive and Strichartz estimates as \eqref{dis:free} and \eqref{str:free} for the wave equation \eqref{eq:wave}. Different from the convention that the operator $H_{a,b}$ in \eqref{H-ab} is often regarded as the Laplacian perturbed by the potential \eqref{potential}, we may regard the operator $H_{a,b}$ as the harmonic oscillator perturbed by an inverse square potential since the spectral properties of $H_{a,b}$ possess more similarity to that of the harmonic oscillator $-\Delta+|x|^2$. Based on this observation, we will, without loss of generality, always take $b=1$ in \eqref{H-ab} or let $H_{a,1}$ take the place of $H_{a,b}$ throughout the paper.

The main effort will be devoted to prove the dispersive estimate \eqref{dis:free} since the Strichartz estimates can be derived from the dispersive estimate via the standard argument of Keel-Tao \cite{KT98}. To obtain the desired dispersive estimate, the steps we shall follow in the sequel consists of proving a dispersive estimate for the Schr\"odinger flow $e^{itH_{a,1}}$, obtaining a Gaussian type upper bound for the heat kernel $e^{-tH_{a,1}}(x,y)$, establishing the Bernstein's inequalities associated to Schr\"odinger operators and finally passing to the classical subordination formula. In order to carry out these steps, let us make some preparations. To obtain the dispersive estimate for the Schr\"odinger propagator $e^{itH_{a,1}}$, a natural way is to write down the Schr\"odinger kernel explicitly, for which a crucial role is played by the spectrum of the Laplace-Beltrami operator $-\Delta_{\mathbb{S}^2}$ on the unit sphere $\mathbb{S}^2$ in $\R^3$. It is well-known that $-\Delta_{\mathbb{S}^2}$ admits a purely discrete spectrum, consisting of real eigenvalues $\mu_k=k(k+1),k\in\mathbb{N}:=\{0,1,2,\cdots\}$ with finite multiplicity $2k+1$, and the set of the associated ($L^2$-normalised) eigenfunctions forms a complete orthonormal basis for $L^2(\mathbb{S}^2;\mathbb{C})$. In view of the symmetry of the wave equation, we shall always require $t>0$ for simplicity.

The Sobolev space associated to the Schr\"odinger operator $H_{a,1}$ is defined by
\begin{equation*}
\dot{\mathcal{H}}^s_{a,1}(\R^3):=H_{a,1}^{-s/2}L^2(\R^3)
\end{equation*}
and then the usual Sobolev space $\dot{H}^s(\R^3)$ (corresponding to the free Laplacian $-\Delta$) can be given by
\begin{equation*}
\dot{H}^s(\R^3):=\dot{\mathcal{H}}^s_{0,0}(\R^3).
\end{equation*}
The homogeneous Besov norm $\|\cdot\|_{\dot{\mathcal{B}}^s_{p,r}(\R^3)}$ is defined by (see also Sect. 5.2)
\begin{equation}\label{Besov}
\|f\|_{\dot{\mathcal{B}}^s_{p,r}(\R^3)}=\left(\sum_{j\in\mathbb{Z}}2^{sjr}\|\psi_j(\sqrt{H_{a,1}})f\|_{L^p(\R^3)}^r\right)^{1/r},
\end{equation}
where $s\in\R, p,r\in[1,\infty]$ and $\psi_j$ is a partition of unity
\begin{equation*}
\sum_{j\in\mathbb{Z}}\psi_j(\lambda)=1,\quad \forall\lambda>0.
\end{equation*}
In particular, one has
\begin{equation*}
\|f\|_{\dot{\mathcal{H}}^s_{a,1}(\R^3)}:=\|H_{a,1}^{s/2}f\|_{L^2(\R^3)}=\left\|\left(\sum_{j\in\mathbb{Z}}2^{2sj}|\psi_j(\sqrt{H_{a,1}})f|^2\right)^{1/2}\right\|_{L^2(\R^3)}
=\|f\|_{\dot{\mathcal{B}}^s_{2,2}(\R^3)}.
\end{equation*}
Now we state the first result concerning the dispersive estimate for the wave equation \eqref{eq:wave}.
\begin{theorem}[Dispersive estimate]\label{thm:wave}
Let $T$ be a constant in $(0,\pi)$, then, for any $t\in(0,T)$, there exists a positive constant $C$ independent of $t$ such that, for all $f\in\dot{\mathcal{B}}^{3/2}_{1,1}(\R^3)$ and $g\in\dot{\mathcal{B}}^{1}_{1,1}(\R^3)$
\begin{equation}\label{dis:wave}
\begin{split}
&\left\|\cos(t\sqrt{H_{a,1}})f\right\|_{L^\infty(\R^3)}\lesssim\frac{1}{\sin t}\|f\|_{\dot{\mathcal{B}}^{3/2}_{1,1}(\R^3)},\\
&\left\|\frac{\sin(t\sqrt{H_{a,1}})}{\sqrt{H_{a,1}}}g\right\|_{L^\infty(\R^3)}\lesssim\frac{1}{\sin t}\|g\|_{\dot{\mathcal{B}}^{1}_{1,1}(\R^3)}.
\end{split}
\end{equation}
\end{theorem}
\begin{remark}
Theorem \ref{thm:wave} is an analog of the main result of \cite{FZZ22}, in which the dispersive estimate \eqref{dis:wave} for the Aharonov-Bohm Hamiltonian (a Schr\"odinger operator with a scaling-critical magnetic potential)
\begin{equation*}
\mathcal{L}_{\mathbf{A}}:=\left(i\nabla+\frac{\mathbf{A}(x/|x|)}{|x|}\right)^2,\quad \mathbf{A}\in W^{1,\infty}(\mathbb{S}^1;\R^2)
\end{equation*}
was obtained by using a suitable representation of the fundamental solution of the wave equation.
Thanks to the appearance of the quadratic potential $|x|^2$ in \eqref{H-ab}, we are allowed to obtain an explicit representation formula for the solution of the Schr\"odinger and heat equations via the spectral theorem since the spectral properties of the Hamiltonian $H_{a,1}$ are already well-understood (see \cite[Proposition 3.2]{FFFP13}). This is not surprising because one can regard the operator \eqref{H-ab} as the harmonic oscillator perturbed by an inverse square potential. In light of the argument of D'Ancona-Pierfelice \cite{DP05}, the key ingredient to prove the desired dispersive estimate \eqref{dis:wave} is to obtain a Gaussian type upper bound for the kernel of the heat semigroup $e^{-tH_{a,1}}$ in the sense of Simon \cite{Sim82}. Similar results based on such a strategy can be found in the previous works \cite{DP05,DPR10,Pie07}.
\end{remark}
As a consequence of Theorem \ref{thm:wave}, we can obtain the Strichartz estimates for the propagator $e^{it\sqrt{H_{a,1}}}$.
\begin{theorem}[Strichartz estimate]\label{thm:str}
Let $H_{a,1}$ be the Hamiltonian given by \eqref{H-ab} with $b=1$ and $a>-\frac{1}{4}$ and let $u$ be given by \eqref{wave:solution}. If the pair $(q,r)$ satisfies the condition \eqref{adm:scaling} and $s$ given by \eqref{cd:gap} belongs to the interval $[0,3/2)$, then there exists some constant $C=C_{a,q,r,s}>0$ independent of $f,g$ such that
\begin{equation}\label{est:str}
\|u\|_{L_t^q\left((0,\pi);L_x^r(\R^3)\right)}\leq C\left(\|f\|_{\dot{\mathcal{H}}^s_{a,1}(\R^3)}+\|g\|_{\dot{\mathcal{H}}^{s-1}_{a,1}(\R^3)}\right).
\end{equation}
\end{theorem}
\begin{remark}
Theorem \ref{thm:str} is new for the wave equation associated to the perturbed harmonic oscillator, although the main result of \cite{DPR10} can standardly yield similar Strichartz estimates for the wave equation associated to the standard harmonic oscillator $-\Delta+|x|^2$. We point out that the Sobolev norm $\dot{\mathcal{H}}^s_{a,1}$ at the right hand side of \eqref{est:str} \emph{can} be replaced by the Sobolev norm $\dot{H}_{osc}^s$ associated with the harmonic oscillator $-\Delta+|x|^2$. Indeed, for dimension $d\geq3$, the Sobolev space $\dot{H}^s_a$ related to the inverse square potential is equivalent to the usual Sobolev space $\dot{H}^s$ associated to the free Laplacian $-\Delta$ (see \cite[Theorem 1.2]{KMVZZ18}). However, the Sobolev norm $\dot{\mathcal{H}}^s_{a,1}$ at the right hand side of \eqref{est:str} \emph{cannot} be replaced by the classical Sobolev space $\dot{H}^s$ since $\dot{H}_{osc}^s$ is \emph{strictly} contained in $\dot{H}^s$ (see \cite[Theorem 3]{BT06}).
\end{remark}

\section{preliminaries}

In this section, we discuss some analytical features of the operator $H_{a,1}$, based on the spectral properties in \cite[Proposition 3.2]{FFFP13}. Throughout the paper, we will always use the notation $C$ to denote a universal positive constant that may vary from line to line.

\noindent The operator $H_{a,1}$, if expressed in the polar coordinates $(r,\vartheta)$, takes the form
\begin{equation}\label{op:polar}
H_{a,1}=-\partial_r^2-\frac{2}{r}\partial_r+\frac{1}{4}r^2+\frac{a-\Delta_{\mathbb{S}^2}}{r^2}.
\end{equation}
In the spherical coordinates $(\theta,\phi)$, we can write $\vartheta:=(\cos\theta\sin\phi,\sin\theta\sin\phi,\cos\phi)$ for $\vartheta\in\mathbb{S}^2$ with $(\theta,\phi)\in[0,2\pi]\times[0,\pi]$ and then $\Delta_{\mathbb{S}^2}$, the Laplace-Beltrami operator on the unit sphere $\mathbb{S}^2$, takes the form
\begin{equation}\label{op:sphere}
\Delta_{\mathbb{S}^2}=\csc^2\phi\partial_\theta^2+\cot\phi\partial_\phi+\partial_\phi^2.
\end{equation}
Let us consider the eigenvalue problem
\begin{equation}\label{pro-1}
\csc^2\phi\partial_\theta^2Y+\cot\phi\partial_\phi Y+\partial_\phi^2Y=-\gamma Y,\quad \lambda\geq0
\end{equation}
and find solutions of the form $Y(\theta,\phi)=\Theta(\theta)\Phi(\phi)$.
By the method of separation of variables, \eqref{pro-1} reduces to the following two ODEs
\begin{align}
&\Theta''(\theta)+\beta\Theta(\theta)=0,\label{pro-11}\\
&\Phi''(\phi)+\cot\phi\Phi'(\phi)+(\mu-\beta\csc^2\phi)\Phi(\phi)=0.\label{pro-12}
\end{align}
The first equation \eqref{pro-11} is easy to handle, whose solution $\Theta\in L^2([0,2\pi];\mathbb{C})$ is given by $\Theta(\theta)=e^{im\theta}$ with the corresponding eigenvalue $\beta=m^2,m\in\mathbb{Z}$. The second equation \eqref{pro-12} is thus reduced to
\begin{equation}\label{pro-2}
\Phi''(\phi)+\cot\phi\Phi'(\phi)+(\mu-m^2\csc^2\phi)\Phi(\phi)=0.
\end{equation}
If we make the change of variables
\begin{equation*}
\xi=\cos\phi,\quad \omega(\xi)=\Phi(\cos\phi),
\end{equation*}
then \eqref{pro-2} is transformed into
\begin{equation}\label{pro-3}
(1-\xi^2)\omega''(\xi)-2\xi\omega'(\xi)+\left(\mu-\frac{m^2}{1-\xi^2}\right)\omega(\xi)=0
\end{equation}
or, equivalently, the self-adjoint form
\begin{equation*}
[(1-\xi^2)\omega'(\xi)]'+\left(\mu-\frac{m^2}{1-\xi^2}\right)\omega(\xi)=0.
\end{equation*}
Note that \eqref{pro-3} is exactly the generalized Legendre equation, whose solution is given by $\Phi(\phi)=P_k^m(\cos\phi)$, the associated Legendre polynomial of degree $k$ and order $m$ with $|m|\leq k$ and the corresponding eigenvalue is $\mu_k=k(k+1)$ with multiplicity $2k+1$. Therefore, we obtain the desired ($L^2$-normalised) eigenfunction of \eqref{pro-1}
\begin{equation}\label{sph:harmonics}
Y_m^k(\vartheta):=Y_m^k(\theta,\phi)=\left(\frac{2m+1}{4\pi}\frac{(k-m)!}{(k+m)!}\right)^{1/2}P_k^m(\cos\phi)e^{im\theta},\quad k,m\in\mathbb{Z},k\geq0,
\end{equation}
which is the usual spherical harmonics on $\mathbb{S}^2$. Associated the spherical harmonics in \eqref{sph:harmonics}, the zonal function $Z_\vartheta^{(k)}(\vartheta')$ is given by
\begin{equation}\label{zonal:sphe}
Z_\vartheta^{(k)}(\vartheta'):=\sum_{m=-k}^kY_m^k(\vartheta)\overline{Y_m^k(\vartheta')}.
\end{equation}
It is well-known that the zonal function $Z_\vartheta^{(k)}(\vartheta')$ has a uniform upper bound (see e.g. \cite{Mul66})
\begin{equation}\label{zonal:up-bd}
|Z_\vartheta^{(k)}(\vartheta')|\leq Z_\vartheta^{(k)}(\vartheta)=\frac{2k+1}{4\pi},\quad \forall k\geq0,\quad \vartheta,\vartheta'\in\mathbb{S}^2.
\end{equation}
Note that the family of eigenfunctions $Y_m^k(\vartheta),k\in\mathbb{N},m\in\mathbb{Z}$ given in \eqref{sph:harmonics} forms a complete orthonormal basis for $L^2(\mathbb{S}^2)$. For a function $\varphi\in L^2(\R^3)$, one can expand $\varphi$ into Fourier series of the form
\begin{align}\label{exp:series}
\varphi(x):=\varphi(r,\vartheta)&=\sum_{k=0}^\infty\sum_{m=-k}^kY_m^k(\theta,\phi)\int_0^{2\pi}\int_0^\pi\varphi(r,\theta',\phi')\overline{Y_m^k(\theta',\phi')}\mathrm{d}\theta'\sin\phi'\mathrm{d}\phi'\nonumber\\
&=\sum_{k=0}^\infty\int_{\mathbb{S}^2}\varphi(r,\vartheta')\overline{Z^{(k)}_\vartheta(\vartheta')}\mathrm{d}\sigma(\vartheta'),
\end{align}
where $Y_m^k(\theta,\phi)$ is given in \eqref{sph:harmonics} and $Z^{(k)}_\vartheta(\vartheta')$ is given in \eqref{zonal:sphe}.
One can always expand an $L^2$ function into a Fourier series in terms of the spherical harmonics \eqref{sph:harmonics} or the zonal function \eqref{zonal:sphe} as \eqref{exp:series}.
In view of \eqref{op:polar}, the action of the operator $H_{a,1}$ on each eigenspace $\mathcal{H}^k:=\mathrm{span}\{Y_m^k:m\in\mathbb{Z},|m|\leq k\}$ (whose dimension is $\dim\mathcal{H}^k=2k+1$) can be given via
\begin{equation}\label{op:polar-2}
H_{a,1}=-\partial_r^2-\frac{2}{r}\partial_r+\frac{1}{4}r^2+\frac{\mu_k+a}{r^2},\quad \mu_k=k(k+1).
\end{equation}
We are now in the right position to give the spectrum of $H_{a,1}$. Let us consider the eigenvalue problem for the operator $H_{a,1}$
\begin{equation}\label{eigen-p}
-\Delta e(x)+\frac{a}{|x|^2}e(x)+\frac{|x|^2}{4}e(x)=\lambda e(x).
\end{equation}
In view of \eqref{exp:series} and \eqref{op:polar-2}, one can obtain the eigenvalues and eigenfunctions of the ODE \eqref{eigen-f} via a suitable transform (see the transform below (3.6) in the proof of \cite[Proposition 3.2]{FFFP13}). Let
\begin{equation}\label{alpha-beta}
\alpha_k=\frac{1}{2}-\sqrt{\left(\frac{1}{2}+k\right)^2+a},\quad \text{and}\quad \beta_k=\sqrt{\left(\frac{1}{2}+k\right)^2+a}
\end{equation}
so that $\alpha_k+\beta_k=\frac{1}{2}$ for $k\in\mathbb{N}$. Then the eigenvalue of the operator $H_{a,1}$ in \eqref{eigen-p} can be given in terms of \eqref{alpha-beta} by
\begin{equation}\label{eigen-v}
\lambda_{m,k}=2m+\frac{3}{2}-\alpha_k=2m+1+\beta_k,\quad m,k\in\mathbb{N}:=\{0,1,2,\cdots\}
\end{equation}
with the corresponding ($L^2$-normalised) eigenfunction
\begin{equation}\label{eigen-f}
e_{m,k}(x)=\left(\frac{m!}{2^{\beta_k}\Gamma(m+1+\beta_k)}\right)^{1/2}|x|^{-\alpha_k}e^{-\frac{|x|^2}{4}}L_m^{\beta_k}\left(\frac{|x|^2}{2}\right)\psi_k\left(\frac{x}{|x|}\right),
\end{equation}
where $\psi_k$ is given via
\begin{equation}\label{psi-k}
\psi_k(\vartheta)\overline{\psi_k(\vartheta')}=\sum_{n=-k}^kY_n^k(\vartheta)\overline{Y_n^k(\vartheta')}
\end{equation}
and $L_m^\alpha(t)$ stands for the generalized Laguerre polynomials
\begin{equation}\label{Laguerre}
L_m^\alpha(t)=\frac{(1+\alpha)_m}{m!}\sum_{k=0}^m\frac{(-m)_k}{(1+\alpha)_k}\frac{t^k}{k!}.
\end{equation}
Note that the zonal function $Z^{(k)}_\vartheta(\vartheta')$ in \eqref{zonal:sphe} can be expressed in terms of $\psi_k$ in \eqref{psi-k} as
\begin{equation*}
Z^{(k)}_\vartheta(\vartheta')=\psi_k(\vartheta)\overline{\psi_k(\vartheta')}.
\end{equation*}
In general, for a well-behaved function $F$, one can define the functional $F(H_{a,1})$ via the spectral theorem
\begin{equation*}
\left(F(H_{a,1})\varphi\right)(x)=\int_{\R^3}\sum_{m,k=0}^\infty F(\lambda_{m,k})e_{m,k}(x)\overline{e_{m,k}(y)}\varphi(y)\mathrm{d}y,
\end{equation*}
where $\lambda_{m,k}$ is given in \eqref{eigen-v} and $e_{m,k}$ is given in \eqref{eigen-f}.
Before closing this section, we shall give a useful lemma concerning the uniform boundedness of the sum of the product of the zonal function in \eqref{zonal:sphe} and the modified Bessel functions of the first kind.
\begin{lemma}\label{lem:sum-ker}
Let $a>-\frac{1}{4}$ and $\beta_k$ be defined in \eqref{alpha-beta}. Define a function $K:\R_+\times\mathbb{S}^2\times\mathbb{S}^2\rightarrow\mathbb{C}$ by
\begin{equation}\label{funct:ker}
K(\rho,\vartheta,\vartheta')=\rho^{-\frac{1}{2}}\sum_{k=0}^\infty i^{-\beta_k}J_{\beta_k}(\rho)Z^{(k)}_\vartheta(\vartheta'),
\end{equation}
where $Z^{(k)}_\vartheta(\vartheta')$ denotes the zonal function in \eqref{zonal:sphe}. Then we have
\begin{align*}
\sup_{\rho\geq0 \atop \vartheta,\vartheta'\in\mathbb{S}^2}|K(\rho,\vartheta,\vartheta')|<+\infty&\quad \text{for}\quad a>0,\\
\sup_{\rho\geq0 \atop \vartheta,\vartheta'\in\mathbb{S}^2}\frac{|K(\rho,\vartheta,\vartheta')|}{1+\rho^{\beta_0-\frac{1}{2}}}<+\infty,&\quad \text{for}\quad -\frac{1}{4}<a<0.
\end{align*}
\end{lemma}
\begin{remark}\label{rem:converge}
We remark that Lemma \ref{lem:sum-ker} is summarized from \cite[Sect. 6]{FFFP13}.
The series expression in \eqref{funct:ker} should be understood locally, i.e. in the sense that there exists some $k_0\geq1$ such that the series
\begin{equation*}
\rho^{-\frac{1}{2}}\sum_{k=k_0+1}^\infty i^{-\beta_k}J_{\beta_k}(\rho)Z^{(k)}_\vartheta(\vartheta),\quad \forall\rho<R
\end{equation*}
is uniformly convergent for any fixed $R$ and
\begin{equation*}
K(\rho,\vartheta,\vartheta')-\rho^{-\frac{1}{2}}\sum_{k=0}^{k_0} i^{-\beta_k}J_{\beta_k}(\rho)Z^{(k)}_\vartheta(\vartheta)\in L_{loc}^\infty(\R_+\times\mathbb{S}^2\times\mathbb{S}^2,\mathbb{C}).
\end{equation*}
\end{remark}
\begin{proof}[Proof of Lemma \ref{lem:sum-ker}]
We give a proof for the benefits of the reader. We only present the proof for the case $a\geq0$ since this is enough for the purpose of the present paper. The proof of the other case $-\frac{1}{4}<a<0$ is similar but a little more complicated, and we may refer the interested readers to \cite[Sect. 6]{FFFP13}.

\noindent Our goal is to show that the function $K$ defined by \eqref{funct:ker} is uniformly bounded, i.e.
\begin{equation*}
\left|\rho^{-\frac{1}{2}}\sum_{k=0}^\infty i^{-\beta_k}J_{\beta_k}(\rho)Z^{(k)}_\vartheta(\vartheta')\right|<+\infty,\quad \forall(\rho,\vartheta,\vartheta')\in\R_+\times\mathbb{S}^2\times\mathbb{S}^2.
\end{equation*}
To this end, we split $K$ as
\begin{align}\label{split}
K&=\rho^{-\frac{1}{2}}\sum_{k=0}^\infty i^{-\beta_k}J_{\beta_k}(\rho)Z^{(k)}_\vartheta(\vartheta')\nonumber\\
&=\rho^{-\frac{1}{2}}\sum_{k=0}^\infty i^{k+\frac{1}{2}}J_{k+\frac{1}{2}}(\rho)Z^{(k)}_\vartheta(\vartheta')\nonumber\\
&\qquad +\rho^{-\frac{1}{2}}\sum_{k=0}^\infty\left(i^{-\beta_k}J_{\beta_k}(\rho)
-i^{k+\frac{1}{2}}J_{k+\frac{1}{2}}(\rho)\right)Z^{(k)}_\vartheta(\vartheta')\nonumber\\
:&=\Pi_1+\Pi_2.
\end{align}
In view of the Jacobi-Anger expansion for the plane waves (see e.g. \cite{Mul66,Wat44})
\begin{equation*}
e^{-ix\cdot y}=(2\pi)^{3/2}(|x||y|)^{-1/2}\sum_{k=0}^\infty i^kJ_{k+\frac{1}{2}}(|x||y|)Z^{(k)}_{x/|x|}(y/|y|),\quad x,y\in\R^3,
\end{equation*}
we get
\begin{equation}\label{JA-exp}
(2\pi)^{-3/2}e^{-i\rho\vartheta\cdot\vartheta'}=\rho^{-1/2}\sum_{k=0}^\infty i^kJ_{k+\frac{1}{2}}(\rho)Z^{(k)}_\vartheta(\vartheta').
\end{equation}
Hence, the first term $\Pi_1$ in \eqref{split} is obviously bounded, in view of \eqref{JA-exp}; actually, one has $\Pi_1=(2\pi i)^{-3/2}e^{-i\rho\vartheta\cdot\vartheta'}i$.

The second term $\Pi_2$ in \eqref{split} is bounded for $\rho\leq\epsilon$ with $\epsilon$ small enough. Indeed, in view of the bound for the zonal function \eqref{zonal:up-bd} and the bound for the Bessel function (see e.g. \cite{Wat44})
\begin{equation*}
|J_\nu(r)|\leq\frac{(r/2)^\nu}{\Gamma(1+\nu)}e^{r^2/4} \quad \forall \nu>0,r\geq0,
\end{equation*}
we have, for $\rho\leq\epsilon$ with $\epsilon$ small enough,
\begin{align*}
|\Pi_2|&=\left|\rho^{-\frac{1}{2}}\sum_{k=0}^\infty i^{-\beta_k}J_{\beta_k}(\rho)Z^{(k)}_\vartheta(\vartheta')\right|\\
&\leq\rho^{-\frac{1}{2}}\sum_{k=0}^\infty\frac{2k+1}{4\pi\Gamma(1+\beta_k)}(\rho/2)^{\beta_k}e^{\rho^2/4}\\
&\leq\frac{e^{\epsilon^2/4}}{4\sqrt{2}\pi}\sum_{k=0}^\infty\frac{2k+1}{\Gamma(1+\beta_k)}
(\rho/2)^{\beta_k-\frac{1}{2}}\\
&\leq\frac{(\rho/2)^{\beta_0-\frac{1}{2}}e^{\epsilon^2/4}}{4\sqrt{2}\pi}
\sum_{k=0}^\infty\frac{2k+1}{\Gamma(1+\beta_k)}\\
&\leq C\rho^{\beta_0-\frac{1}{2}},
\end{align*}
where $C=C(\epsilon)>0$ is a constant independent of $\rho,\vartheta,\vartheta'$.

\noindent For $\rho>\epsilon$, we use the contour integral representation for the Bessel function
\begin{equation*}
J_\nu(r)=\frac{1}{2\pi i}\int_\Gamma e^{\frac{r}{2}(z-\frac{1}{z})}\frac{\mathrm{d}z}{z^{1+\nu}}
\end{equation*}
to express the second term $\Pi_2$ in \eqref{split}
\begin{align}\label{int:cont}
\Pi_2&=\rho^{-\frac{1}{2}}\sum_{k=0}^\infty\left(i^{-\beta_k}J_{\beta_k}(\rho)
-i^{k+\frac{1}{2}}J_{k+\frac{1}{2}}(\rho)\right)Z^{(k)}_\vartheta(\vartheta')\\
=&\frac{\rho^{-1/2}}{2\pi i}\int_\Gamma e^{\frac{\rho}{2}(z-\frac{1}{z})}\left(\sum_{k=0}^\infty\left((iz)^{k+\frac{1}{2}-\beta_k}-1\right)
\frac{Z^{(k)}_\vartheta(\vartheta')}{(iz)^{k+\frac{1}{2}}}\right)\frac{\mathrm{d}z}{z},\nonumber
\end{align}
where $\Gamma\subset\mathbb{C}$ denotes the positive oriented contour running along the straight line from negative infinity to $z=-1$ below the negative real axis, the counterclockwise oriented unit circle from $z=-1$ to itself and the straight line from $z=-1$ to negative infinity above the negative real axis.
For convenience of statements, we denote by $\Gamma_1$ the counterclockwise oriented circumference of radius 1 and by $\Gamma_2$ the two straight lines running between $z=-\infty$ and $z=-1$, and we split the integral in \eqref{int:cont} accordingly, i.e.
\begin{equation*}
\int_\Gamma=\int_{\Gamma_1}+\int_{\Gamma_2}:=I_1+I_2.
\end{equation*}
In view of the analyticity of the integrand of \eqref{int:cont} outside $z=\R_-+i0\pm$, we can write the integral in \eqref{int:cont} more rigorously
\begin{equation*}
\int_\Gamma=\int_{\Gamma_1}+\int_{\Gamma_2}=\lim_{\varepsilon\rightarrow0+}
\left(\int_{\Gamma_1^\varepsilon}+\int_{\Gamma_2^\varepsilon}\right),
\end{equation*}
where $\Gamma_1^\varepsilon$ denotes the circumference of radius $1+\varepsilon$ centered at the origin and $\Gamma_2^\varepsilon$ denotes the lines running along $(-\infty,-1-\varepsilon)+i0\pm$.
Note that we have exchanged the order of sum and integral in \eqref{int:cont}, which is permitted for any $\rho,\vartheta,\vartheta'$. Indeed, since $|z|>1$ for all $z\in\Gamma_1^\varepsilon\cup\Gamma_2^\varepsilon$, one has the absolute convergence of the contour integral along $\Gamma_1^\varepsilon\cup\Gamma_2^\varepsilon$ for any $\rho,\vartheta,\vartheta'$ and hence the exchange of the order of sum and integral in \eqref{int:cont} is permitted by Fubini's theorem.

Let us consider $I_1$ (i.e. the integral along $\Gamma_1$) firstly. In view of
\begin{equation*}
\beta_k-(k+\frac{1}{2})=\sqrt{\left(k+\frac{1}{2}\right)^2+a}-(k+\frac{1}{2})=\frac{a}{2k+1}+O(k^{-3}),
\end{equation*}
we have
\begin{align}\label{3-part}
(iz)^{k+\frac{1}{2}-\beta_k}-1
&=-\frac{a}{2k+1}\log(iz)+\frac{a^2}{2}\frac{(\log(iz))^2}{(2k+1)^2}+\frac{O(1)}{k^3}\nonumber\\
&:=I_{11}(z,k)+I_{12}(z,k)+I_{13}(z,k)\quad \text{as}\quad k\rightarrow+\infty
\end{align}
uniformly in $z\in\Gamma_1$. Note that $z^{-\beta_k}$ and $z^{k+\frac{1}{2}}$ admit a branch-cut at the negative real axis of the complex plane (i.e. at $z\in\R_-$), then the function $\log(iz)$ (and $\sqrt{iz}$, etc.) will admit a branch-cut at $z\in\R_-$ as well. Hence, we can, taking \eqref{3-part} into consideration, write $I_1$ (i.e. the integral over $\Gamma_1$ in \eqref{int:cont}) as
\begin{align}\label{3-term:123}
I_1&=\frac{\rho^{-1/2}}{2\pi i}\int_{\Gamma_1} e^{\frac{\rho}{2}(z-\frac{1}{z})}\left(\sum_{k=0}^\infty\left(I_{11}(z,k)+I_{12}(z,k)+I_{13}(z,k)\right)
\frac{Z^{(k)}_\vartheta(\vartheta')}{(iz)^{k+\frac{1}{2}}}\right)\frac{\mathrm{d}z}{z}\nonumber\\
&:=\mathcal{I}_{11}+\mathcal{I}_{12}+\mathcal{I}_{13},
\end{align}
where every summand $\mathcal{I}_{1j},j=1,2,3$ corresponds to the integrand with the corresponding $I_{1j},j=1,2,3$. Since $|z|\equiv1$ for all $z\in\Gamma_1$, we have $|e^{\frac{\rho}{2}(z-\frac{1}{z})}|=1$ and thus, in view of \eqref{zonal:up-bd} again, we obtain
\begin{equation*}
|\mathcal{I}_{13}|\lesssim\rho^{-1/2}\sum_{k=0}^\infty\frac{2k+1}{k^3}\lesssim\epsilon^{-1/2},\quad \text{for}\quad \rho>\epsilon,
\end{equation*}
which indicates that $|\mathcal{I}_{13}|$ is bounded uniformly in $\rho,\vartheta,\vartheta'$.
From the definitions of the spherical harmonics \eqref{sph:harmonics} and the zonal functions \eqref{zonal:sphe}, one can recover the well-known identity (see \cite{Mul66})
\begin{equation}\label{ident:sz}
Z^{(k)}_\vartheta(\vartheta')=\frac{2k+1}{4\pi}P_k(\vartheta\cdot\vartheta')
\end{equation}
where $P_k$ is the Legendre polynomial of degree $k$ (a solution of the Legendre equation \eqref{pro-3}).
In view of the identity \eqref{ident:sz} and the generating function formula for the Legendre polynomials (see e.g. \cite[(22.9.12)]{AS65})
\begin{equation}\label{gen:poly}
\sum_{k=0}^\infty P_k(s)r^k=(1-2rs+r^2)^{-1/2},\quad |r|<1,
\end{equation}
we have, for $\mathcal{I}_{11}$,
\begin{align}\label{use-form}
&-a\log(iz)\sum_{k=0}^\infty\frac{Z^{(k)}_\vartheta(\vartheta')}{2k+1}(iz)^{-(k+\frac{1}{2})}\nonumber\\
&=-\frac{a\log(iz)}{4\pi}\sum_{k=0}^\infty P_k(\vartheta\cdot\vartheta')(iz)^{-(k+\frac{1}{2})}\nonumber\\
&=-\frac{a(iz)^{-1/2}\log(iz)}{4\pi\sqrt{1+2iz^{-1}\vartheta\cdot\vartheta'-z^{-2}}}\nonumber\\
&=-\frac{a(-iz)^{1/2}\log(iz)}{4\pi\sqrt{z^2+2iz(\vartheta\cdot\vartheta')-1}},\quad |z|>1
\end{align}
and hence
\begin{equation}\label{app:1st-term}
\mathcal{I}_{11}=-\frac{a\rho^{-1/2}}{8\pi^2i}\lim_{\varepsilon\rightarrow0+}\int_{\Gamma_1^\varepsilon} e^{\frac{\rho}{2}(z-\frac{1}{z})}\frac{(-iz)^{1/2}\log(iz)}{\sqrt{z^2+2iz(\vartheta\cdot\vartheta')-1}}
\frac{\mathrm{d}z}{z}.
\end{equation}
Note that $|e^{\frac{\rho}{2}(z-\frac{1}{z})}|\equiv1$ for all $\rho>0$ and $z\in\Gamma_1$. If $\vartheta\cdot\vartheta'$ stays far away from $\pm1$, then $1-(\vartheta\cdot\vartheta')^2>\epsilon$ and the term $\mathcal{I}_{11}$ in \eqref{app:1st-term} is uniformly bounded with respect to $\varepsilon\rightarrow0+,\rho>\epsilon$, due to the integrability of the two singularities of square root type of the integrand at $z_{\pm}=-i(\vartheta\cdot\vartheta')\pm\sqrt{1-(\vartheta\cdot\vartheta')^2}$.
If $\vartheta\cdot\vartheta'=\pm1$, one sees that these two singularities at $z_{\pm}$ collapse into a stronger singularity at $z=\pm i$. Let us assume $\vartheta\cdot\vartheta'=-1$ (the other case $\vartheta\cdot\vartheta'=1$ can be similarly treated) for concreteness, then we obtain
\begin{align}\label{PS:formula}
\lim_{\varepsilon\rightarrow0+}&\int_{\Gamma_1^\varepsilon}e^{\frac{\rho}{2}(z-\frac{1}{z})}
\frac{(-iz)^{1/2}\log(iz)}{z-i}\frac{\mathrm{d}z}{z}\\
&=i\pi^2e^{i\rho}+\mathrm{P.V.}\int_{\Gamma_1}e^{\frac{\rho}{2}(z-\frac{1}{z})}
\frac{(-iz)^{1/2}\log(iz)}{z-i}\frac{\mathrm{d}z}{z},
\end{align}
which is simply the Plemelj-Sokhotskyi formula (see e.g. \cite{AF03}) for the limit of Cauchy integrals when tending to a singular point. The second term at the right hand side of \eqref{PS:formula} is a singular integral of the function $e^{\frac{\rho}{2}(z-\frac{1}{z})}(iz)^{-1/2}\log(iz)$, which is differentiable with respect to $z=e^{i\theta}$ for $\theta$ varying in some neighborhood of $\frac{\pi}{2}$ (remember that the discontinuity of the argument of $z$ is along the negative real axis of the complex plane).
Due to the boundedness of the principal value of a Cauchy integral of a differentiable function (see \cite{AF03}), we conclude the boundedness of the term $\mathcal{I}_{11}$ for all $\rho>\epsilon$.
We stress that it is still possible for the principal value integral to be divergent as $\rho\rightarrow\infty$ despite its boundedness for any $\rho>\epsilon$. To exclude this undesired possibility, let us consider the neighborhood of $z=i$ in $\Gamma_1$
\begin{equation*}
\Gamma_1^{\eta}:=\{z=ie^{i\theta}:|\theta|<\eta\ll1\}.
\end{equation*}
In view of the fact
\begin{equation*}
\frac{(\pi+\theta)e^{i\theta/2}}{i(e^{i\theta}-1)}=-\frac{\pi}{\theta}+O(1)\quad \text{as}\quad \theta\rightarrow0,
\end{equation*}
we compute the integral over $\Gamma_1^{\eta}$ for any $\rho\gg1$
\begin{align*}
\mathrm{P.V.}\int_{-\eta}^\eta e^{i\rho\cos\theta}\frac{(\pi+\theta)e^{i\theta/2}}{i(e^{i\theta}-1)}\mathrm{d}\theta
&=\mathrm{P.V.}\int_{-\eta}^\eta e^{i\rho\cos\theta}\left(-\frac{\pi}{\theta}+O(1)\right)\mathrm{d}\theta\\
&=O(1)\mathrm{P.V.}\int_{-\eta}^\eta e^{i\rho\cos\theta}\mathrm{d}\theta-
\pi\mathrm{P.V.}\int_{-\eta}^\eta\frac{e^{i\rho\cos\theta}}{\theta}\mathrm{d}\theta.
\end{align*}
The second principal value integral obviously vanishes (due to the parity of the integrand), i.e.
\begin{equation*}
\mathrm{P.V.}\int_{-\eta}^\eta\frac{e^{i\rho\cos\theta}}{\theta}\mathrm{d}\theta\equiv0,
\end{equation*}
it thus follows (note that the first principal value integral is finite) that
\begin{equation*}
\mathrm{P.V.}\int_{-\eta}^\eta e^{i\rho\cos\theta}\frac{(\pi+\theta)e^{i\theta/2}}{i(e^{i\theta}-1)}\mathrm{d}\theta=O(1).
\end{equation*}
Therefore, the integral along $\Gamma_1^{\eta}$ is uniformly bounded for all $\rho$. The integral over $\Gamma_1\setminus\Gamma_1^{\eta}$ is uniformly bounded as well since $|e^{\frac{\rho}{2}(z-\frac{1}{z})}|\equiv1$ for all $z\in\Gamma_1$. Hence, the principal value integral over $\Gamma_1$ is bounded uniformly in $\rho,\vartheta,\vartheta'$.
If the singularities $z_{\pm}=-i(\vartheta\cdot\vartheta')\pm\sqrt{1-(\vartheta\cdot\vartheta')^2}$ are very close to each other (i.e. $|\vartheta\cdot\vartheta'|>1-\epsilon$), then, similarly as \eqref{PS:formula}, one can split the integral over $\Gamma_1$ into two parts
\begin{equation}\label{sp:arc}
\int_{\Gamma_1}=\int_{\overbrace{z_-,z_+}}+\int_{\Gamma_1\setminus\overbrace{z_-,z_+}},
\end{equation}
where $\overbrace{z_-,z_+}$ denotes the small arc of $\Gamma_1$ between $z_-$ and $z_+$. The second integral at the right hand side of \eqref{sp:arc} can be easily bounded just like the principal value integral above and yields the same uniform boundedness with respect to $\vartheta,\vartheta'$. It is easy to see that there exists some $\hat{\theta}\in\R$ such that
\begin{equation*}
z_{\pm}=-i(\vartheta\cdot\vartheta')\pm\sqrt{1-(\vartheta\cdot\vartheta')^2}=ie^{\pm i\hat{\theta}},
\end{equation*}
then the first term at the right hand side of \eqref{sp:arc} becomes the integral of
\begin{align*}
&\frac{(\pi+\theta)e^{i(\rho\cos\theta+\frac{\theta}{2})}}{\sqrt{(ie^{i\theta}-ie^{-i\hat{\theta}})(ie^{i\theta}-ie^{i\hat{\theta}})}}\\
&=\frac{\pi e^{i\rho\cos\theta}}{\sqrt{|\theta+\hat{\theta}||\theta-\hat{\theta}|}}
\left(1+O(|\theta-\hat{\theta}|)+O(|\theta+\hat{\theta}|)\right)\times
\begin{cases}
-1&\quad\text{if}\quad \theta<-\hat{\theta}\\
1&\quad\text{if}\quad\theta>\hat{\theta}\\
-i&\quad\text{if}\quad -\hat{\theta}<\theta<\hat{\theta}
\end{cases}\\
&=\frac{\pi e^{i\rho\cos\theta}\varphi(\theta)}{\sqrt{|\theta+\hat{\theta}||\theta-\hat{\theta}|}}
+O(|\theta-\hat{\theta}|^{-1/2})+O(|\theta+\hat{\theta}|^{-1/2}),
\end{align*}
where $\varphi$ is a suitable function satisfying $|\varphi(\theta)|\equiv1$.
Hence, we have
\begin{align*}
\left|\int_{\overbrace{z_-,z_+}}\right|&=\left|\pi\int_{-\hat{\theta}}^{\hat{\theta}}
\frac{e^{i\rho\cos\theta}}{\sqrt{\hat{\theta}^2-\theta^2}}\mathrm{d}\theta+O(1)\right|\\
&=\left|\pi\int_{-1}^1
\frac{e^{i\rho\cos(\hat{\theta}s)}}{\sqrt{1-s^2}}\mathrm{d}s+O(1)\right|\\
&\lesssim1\quad \text{uniformly with respect to}\quad \rho,\hat{\theta}.
\end{align*}
Therefore, the integral over $\Gamma_1$ is bounded uniformly in $\rho$ and $\epsilon$.

Finally, it remains to treat the term $\mathcal{I}_{12}$ in \eqref{3-term:123}
\begin{equation*}
\frac{a^2\rho^{-1/2}}{4\pi i}\int_{\Gamma_1}e^{\frac{\rho}{2}(z-\frac{1}{z})}\left(\sum_{k=0}^\infty
\frac{Z^{(k)}_\vartheta(\vartheta')}{(2k+1)^2}(iz)^{-(k+\frac{1}{2})}\right)(\log(iz))^2\frac{\mathrm{d}z}{z}.
\end{equation*}
Let
\begin{align*}
G(z,\vartheta,\vartheta'):&=\sum_{k=0}^\infty
\frac{Z^{(k)}_\vartheta(\vartheta')}{(2k+1)^2}(iz)^{-(k+\frac{1}{2})},\\
g(z,\vartheta,\vartheta'):&=\frac{1}{2i}\sum_{k=0}^\infty
\frac{Z^{(k)}_\vartheta(\vartheta')}{2k+1}(iz)^{-(k+\frac{3}{2})},
\end{align*}
then one easily verifies
\begin{equation*}
G'_z(z,\vartheta,\vartheta')=g(z,\vartheta,\vartheta').
\end{equation*}
In view of the formulas \eqref{ident:sz} and \eqref{gen:poly}, we have
\begin{equation*}
g(z,\vartheta,\vartheta')=\frac{(iz)^{-3/2}}{8\pi i\sqrt{1+2iz^{-1}(\vartheta\cdot\vartheta')-z^{-2}}}.
\end{equation*}
Since $g(z,\vartheta,\vartheta')$ has a square root type singularity when $\vartheta\cdot\vartheta'\neq-1$ or has a $(z-i)^{-1}$ type singularity at $z=i$ when $\vartheta\cdot\vartheta'=-1$ (we only consider this case here, the other case $\vartheta\cdot\vartheta'=1$ is similar as before), we conclude that $G(z,\vartheta,\vartheta')$ has at most a log-type singularity, which is integrable, and hence the corresponding integral is uniformly bounded. Therefore, the term $I_1$ (i.e. the integral \eqref{int:cont} over $\Gamma_1$) is uniformly bounded.

We proceed to consider the term $I_2$
\begin{equation}\label{I-2}
I_2=\frac{\rho^{-1/2}}{2\pi i}\int_{\Gamma_2} e^{\frac{\rho}{2}(z-\frac{1}{z})}\left(\sum_{k=0}^\infty\left((iz)^{k+\frac{1}{2}-\beta_k}-1\right)
\frac{Z^{(k)}_\vartheta(\vartheta')}{(iz)^{k+\frac{1}{2}}}\right)\frac{\mathrm{d}z}{z}.
\end{equation}
After changing variables $z=-e^s$, exchanging the order of sum and integral and rearranging terms subsequently, we can rewrite $I_2$ of \eqref{I-2} (similar to $I_1$) as
\begin{equation}\label{I-2:split}
\begin{split}
I_2&=\frac{\rho^{-1/2}}{2\pi i}\sum_{k=0}^\infty Z^{(k)}_\vartheta(\vartheta')\left(\Phi_k(\rho)+\Psi_k(\rho)\right)\\
&:=\mathcal{I}_{21}+\mathcal{I}_{22},
\end{split}
\end{equation}
where
\begin{align*}
\Phi_k(\rho) & =2i^{-(1+\beta_k)}\sin(\pi\beta_k)\int_0^\infty e^{-\rho\sinh s}(e^{-s\beta_k}-e^{-(k+\frac{1}{2})s})\mathrm{d}s, \\
\Psi_k(\rho) & =-2i\int_0^\infty e^{-\rho\sinh s-(k+\frac{1}{2})s}(i^{-\beta_k}\sin(\pi\beta_k)-i^{-(k+\frac{1}{2})}\sin\pi(k+\frac{1}{2}))\mathrm{d}s.
\end{align*}
In view of \eqref{zonal:up-bd} again, we obtain
\begin{align*}
|\mathcal{I}_{21}|&=\left|\frac{\rho^{-1/2}}{2\pi i}\sum_{k=0}^\infty Z^{(k)}_\vartheta(\vartheta')\Phi_k(\rho)\right| \\
&\lesssim\rho^{-1/2}\sum_{k=0}^\infty(2k+1)\int_0^\infty e^{-\rho\sinh s}\left|e^{-s\beta_k}-e^{-(k+\frac{1}{2})s}\right|\mathrm{d}s\\
&\lesssim\rho^{-1/2}\sum_{k=0}^\infty(2k+1)\left|\frac{1}{\beta_k}-\frac{1}{k+\frac{1}{2}}\right|
\lesssim\rho^{-1/2}.
\end{align*}
To address $\mathcal{I}_{22}$ in \eqref{I-2:split}, we use the fact
\begin{align*}
&i^{-\beta_k}\sin(\pi\beta_k)-i^{-(k+\frac{1}{2})}\sin\pi(k+\frac{1}{2})\\
&=(i^{-\beta_k}-i^{-(k+\frac{1}{2})})\sin(\pi\beta_k)
+i^{-(k+\frac{1}{2})}\left(\sin(\pi\beta_k)-\sin\pi(k+\frac{1}{2})\right)\\
&=i^{-(k+\frac{1}{2})}(-1)^{k+1}\left(\frac{a\pi i}{2(2k+1)}+\frac{(a\pi)^2}{8(2k+1)^2}\right)
+i^{-(k+\frac{1}{2})}(-1)^{k+1}\frac{(a\pi)^2}{2(2k+1)^2}+O(k^{-3})\\
&=i^{k+\frac{5}{2}}\frac{a\pi}{2(2k+1)}+i^{k+\frac{3}{2}}\frac{5\pi^2a^2}{8(2k+1)^2}+O(k^{-3})
\quad\text{as}\quad k\rightarrow+\infty
\end{align*}
and the formula \eqref{gen:poly} to obtain
\begin{align*}
|\mathcal{I}_{22}|&=\left|\frac{\rho^{-1/2}}{2\pi i}\sum_{k=0}^\infty Z^{(k)}_\vartheta(\vartheta')\Psi_k(\rho)\right|
\leq\Bigg|\rho^{-1/2}\int_0^\infty e^{-\rho\sinh s-\frac{s}{2}}\\
&\times\Bigg[\frac{a\pi}{4i^{3/2}\sqrt{1-ie^{-s}(\vartheta\cdot\vartheta')-e^{-2s}}}\\
&+\frac{5\pi^2a^2Z^{(k)}_\vartheta(\vartheta')}{16i^{-(k+\frac{3}{2})}e^{ks}}
\left((2k+1)^{-2}+O(k^{-3})\right)\Bigg]\Bigg|.
\end{align*}
By \eqref{zonal:up-bd} and
\begin{equation*}
\frac{1}{\sqrt{1-ie^{-s}(\vartheta\cdot\vartheta')-e^{-2s}}}\lesssim1+s^{-1/2},
\end{equation*}
we see
\begin{align*}
\Bigg|\sum_{k=0}^\infty&i^ke^{-ks}Z^{(k)}_\vartheta(\vartheta')\left((2k+1)^{-2}+O(k^{-3})\right)\Bigg|\\
&\lesssim\sum_{k=0}^\infty(ke^{ks})^{-1}\lesssim\log|s|\quad\text{uniformly in}\quad \vartheta,\vartheta'.
\end{align*}
Hence, we can bound $\mathcal{I}_{22}$ in \eqref{I-2:split} by
\begin{equation*}
|\mathcal{I}_{22}|\lesssim\rho^{-1/2}\int_0^\infty e^{-\rho\sinh s-\frac{s}{2}}(1+s^{-1/2}+\log|s|)\mathrm{d}s\lesssim\rho^{-1/2}.
\end{equation*}
Remember that we have $\rho>\epsilon$ for some $\epsilon>0$ in this case, the uniform boundedness of $I_2$ follows.

\noindent Therefore, we obtain the desired bound
\begin{equation*}
\sup_{\rho\geq0 \atop \vartheta,\vartheta'\in\mathbb{S}^2}|K(\rho,\vartheta,\vartheta')|<+\infty
\end{equation*}
and the proof of Lemma \ref{lem:sum-ker} is finished.
\end{proof}

\section{dispersive estimate for Schr\"odinger flow}

In this section, we will construct a representation formula for the kernel of the Schr\"odinger propagator $e^{-itH_{a,1}}$ and then prove a dispersive inequality for this kernel. Such a dispersive inequality will be used to prove the main theorem (Theorem \ref{thm:wave}) of this paper, as mentioned in the introduction.

\begin{proposition}[Schr\"odinger kernel]\label{prop:S-ker}
Let $H_{a,1}$ be the operator given by \eqref{H-ab} and let $u$ be the unique solution to the Schr\"odinger equation
\begin{equation*}
\begin{cases}
i\partial_tu(t,x)=H_{a,1}u(t,x),\\
u(0,x)=f(x)\in L^2(\R^3),
\end{cases}
\end{equation*}
then the solution $u$ can be represented as
\begin{equation}\label{sol-rep:S}
u(t,x)=\left(e^{-itH_{a,1}}f\right)(x):=\int_{\R^3}K^S(x,y)f(y)\mathrm{d}y,
\end{equation}
where $K^S(x,y)$ corresponds to the kernel of the Schr\"odinger propagator $e^{-itH_{a,1}}$.

Moreover, if we write $K_t^S(r_1,\vartheta_1;r_2,\vartheta_2)$ for the kernel $K^S(x,y)$ in polar coordinates $(r,\vartheta)$, then we have
\begin{equation}\label{ker:S}
K_t^S(r_1,\vartheta_1;r_2,\vartheta_2)=\frac{e^{-\frac{r_1^2+r_2^2}{4i\tan t}}}{2i\sqrt{r_1r_2}\sin t}
\sum_{k=0}^\infty I_{\beta_k}\Bigg(\frac{r_1r_2}{2i\sin t}\Bigg)Z_{\vartheta_1}^{(k)}(\vartheta_2),
\end{equation}
where $Z_{\vartheta}^{(k)}(\vartheta')$ is given in \eqref{zonal:sphe} and $\beta_k$ in \eqref{alpha-beta}.
\end{proposition}
\begin{remark}
The integral at the right hand side of \eqref{sol-rep:S} should be understood in the sense of improper integrals, i.e.
\begin{equation*}
u(t,x)=\lim_{R\rightarrow+\infty}\int_{B_R(0)}K^S(x,y)f(y)\mathrm{d}y,\quad \text{where}\quad B_R(0):=\{y\in\R^3:|y|<R\},
\end{equation*}
which is consistent with the statement of Remark \ref{rem:converge}.
\end{remark}
\begin{proof}[Proof of Proposition \ref{prop:S-ker}]
Since the spectrum of the operator $H_{a,1}$ consists of pure points, we have, by the spectral theorem,
\begin{equation}\label{spec:1}
e^{-itH_{a,1}}(x,y)=\sum_{m,k\in\mathbb{N}}e^{-it\lambda_{m,k}}e_{m,k}(x)\overline{e_{m,k}(y)}.
\end{equation}
Substituting the eigenvalues \eqref{eigen-v} and eigenfunctions \eqref{eigen-f} into \eqref{spec:1}, one gets
\begin{equation}\label{spec:2}
\begin{split}
K_t^S(r_1,\vartheta_1;r_2,\vartheta_2)&=(r_1r_2)^{-1/2}e^{-\frac{r_1^2+r_2^2}{4}}\sum_{k=0}^\infty\frac{(r_1r_2)^{\beta_k}e^{-it(1+\beta_k)}}{2^{\beta_k}}\psi_k(\vartheta_1)\overline{\psi_k(\vartheta_2)}\\
&\qquad\qquad\times\sum_{m=0}^\infty e^{-2imt}\frac{m!L_m^{\beta_k}\left(\frac{r_1^2}{2}\right)L_m^{\beta_k}\left(\frac{r_2^2}{2}\right)}{\Gamma(m+1+\beta_k)},
\end{split}
\end{equation}
where $\psi_k$ is defined in \eqref{psi-k}.
From the Poisson kernel formula for the generalized Laguerre polynomials (see e.g. \cite[(6.2.25)]{AAR99}): for $|r|<1,\alpha>-1$
\begin{equation}\label{formula:sum}
\sum_{m=0}^\infty\frac{m!L_m^\alpha(x)L_m^\alpha(y)r^m}{\Gamma(1+\alpha+m)}=\frac{e^{-\frac{r}{1-r}(x+y)}}{(1-r)(xyr)^{\alpha/2}}I_\alpha\left(\frac{2\sqrt{xyr}}{1-r}\right),
\end{equation}
we obtain
\begin{equation}\label{spec:3}
\begin{split}
K_t^S(r_1,\vartheta_1;r_2,\vartheta_2)=\frac{(r_1r_2)^{-1/2}e^{-\frac{r_1^2+r_2^2}{4i\tan t}}}{2i\sin t}
\sum_{k=0}^\infty\psi_k(\vartheta_1)\overline{\psi_k(\vartheta_2)}I_{\beta_k}\Bigg(\frac{r_1r_2}{2i\sin t}\Bigg).
\end{split}
\end{equation}
Finally, in view of \eqref{psi-k} and \eqref{zonal:sphe}, the desired formula \eqref{ker:S} follows.
\end{proof}
\begin{remark}
Note that in the representation formula of the Schr\"odinger kernel \eqref{ker:S} in Proposition \ref{prop:S-ker}, there appears a sine factor $\sin t$, which implies that the time points $t=k\pi$ with $k\in\mathbb{Z}$ have to be excluded. This phenomenon is natural due to the introduction of the Hermite potential $|x|^2$, which implies that the effect of the Hermite potential on the Laplacian is stronger than that of the inverse square potential $|x|^{-2}$, or in other words, the effect of the inverse square potential $|x|^{-2}$ can be viewed as a "small" perturbation of the harmonic oscillator (in the sense that the discreteness of the spectrum is preserved). Hence, the dispersive bound for the kernel \eqref{ker:S} will be of the form $|\sin t|^{-\frac{3}{2}}$ (see Proposition \ref{prop:S-dis} below) with $t\neq k\pi,k\in\mathbb{Z}$, which is different from that for the free Schr\"odinger kernel case $|t|^{-\frac{3}{2}}$, but consistent with that for the (unperturbed) harmonic oscillator (see e.g. \cite{KT05}).
\end{remark}
\noindent In the following, we prove a dispersive bound for the Schr\"odinger kernel obtained in Proposition \ref{prop:S-ker}.
\begin{proposition}\label{prop:S-dis}
Let $H_{a,1}$ be the operator in \eqref{H-ab}, then there exists some constant $C>0$ such that
\begin{equation}\label{dis:S}
\|e^{-itH_{a,1}}\|_{L^1(\mathbb{R}^3)\rightarrow L^\infty(\mathbb{R}^3)}\leq C|\sin t|^{-\frac{3}{2}},\quad \forall t\neq k\pi,\quad k\in\mathbb{Z}.
\end{equation}
\end{proposition}
\begin{remark}
Interpolating between the dispersive estimate \eqref{dis:S} and the unitary property of the Schr\"odinger propagator $e^{-itH_{a,1}}$
\begin{equation*}
\|e^{-itH_{a,1}}f\|_{L^2(\R^3)}=\|f\|_{L^2(\R^3)},
\end{equation*}
one immediately gets the time decay estimate
\begin{equation*}
\|e^{-itH_{a,1}}f\|_{L^{p'}(\R^3)}\lesssim|\sin t|^{-3(\frac{1}{2}-\frac{1}{p'})}\|f\|_{L^p(\R^3)},\quad p\in[1,2],\quad \frac{1}{p}+\frac{1}{p'}=1,
\end{equation*}
which is an analog of \cite[Theorem 1.11 (i)]{FFFP13} for $e^{-itH_{a,0}}$.
\end{remark}
\begin{proof}[Proof of Proposition \ref{prop:S-dis}]
The proof follows immediately from Lemma \ref{lem:sum-ker} and we sketch it here for completeness.
Since the sine function is periodic, we may, without loss of generality, assume $t\in(0,\pi)$ for simplicity.
Note that the Schr\"odinger kernel $e^{-itH_{a,1}}(x,y)$ in \eqref{ker:S} can be rewritten as
\begin{equation*}
e^{-itH_{a,1}}(x,y)=-i(2\sin t)^{-\frac{3}{2}}e^{-\frac{|x|^2+|y|^2}{4i\tan t}}K\left(\frac{|xy|}{2\sin t},\frac{x}{|x|},\frac{y}{|y|}\right),
\end{equation*}
where $K(\cdot,\cdot,\cdot)$ is the function defined in Lemma \ref{lem:sum-ker}. Then, the desired dispersive bound \eqref{dis:S} follows by Lemma \ref{lem:sum-ker}.
Indeed, let us verify \eqref{dis:S} according to $a=0$ or $a>0$. If $a=0$, then we are done since in this case the result reduces to the dispersive estimate for the standard harmonic oscillator (see \cite{KT05}). If $a>0$, we have
\begin{align*}
\|e^{-itH_{a,1}}\|_{L^1(\mathbb{R}^3)\rightarrow L^\infty(\mathbb{R}^3)}&\leq|e^{-itH_{a,1}}(x,y)|\\
&\leq|\sin t|^{-\frac{3}{2}}\left|K\left(\frac{|xy|}{2\sin t},\frac{x}{|x|},\frac{y}{|y|}\right)\right|\\
&\leq C|\sin t|^{-\frac{3}{2}},
\end{align*}
where we use Lemma \ref{lem:sum-ker} in the last inequality. Hence, the proof of the dispersive inequality \eqref{dis:S} is completed.
\end{proof}
\begin{remark}\label{rem:negative}
The parameter $a$ in the inverse square potential $\frac{a}{|x|^2}$ is usually assumed to satisfy $a\geq-\frac{1}{4}$ ($a\geq-\frac{(d-2)^2}{4}$ for general dimension $d\geq2$) due to the sharp Hardy inequality (see e.g. \cite{Hard20})
\begin{equation*}
\int_{\R^d}\frac{|u(x)|^2}{|x|^2}\mathrm{d}x\leq\frac{4}{(d-2)^2}\int_{\R^d}|\nabla u(x)|^2\mathrm{d}x,\quad d\geq3.
\end{equation*}
However, we require $a\geq0$ throughout the paper. We do not consider the case $-\frac{1}{4}\leq a<0$ in the present paper, in light of the result of \cite[Theorem 1.11]{FFFP13}, which says that, for $\frac{1}{p}+\frac{1}{p'}=1,p\in[1,2]$, we have
\begin{equation}\label{dis:inverse}
\|e^{it(\Delta-a|x|^{-2})}f\|_{L^{p'}(\R^3)}\lesssim|t|^{-3(\frac{1}{2}-\frac{1}{p'})}\|f\|_{L^p(\R^3)}, \quad\text{when}\quad a\geq0
\end{equation}
and
\begin{equation}\label{w-dis:inverse}
\|e^{it(\Delta-a|x|^{-2})}f\|_{L^{p',\alpha_1}(\R^3)}\lesssim
\frac{(1+|t|^{\alpha_1})^{1-\frac{2}{p'}}}{|t|^{3(\frac{1}{2}-\frac{1}{p'})}}\|f\|_{L^{p,\alpha_1}(\R^3)},
 \text{when} -\frac{1}{4}<a<0,
\end{equation}
where $\alpha_1=\frac{1}{2}-\sqrt{\frac{1}{4}+a}$ is a strictly positive number and $\|\cdot\|_{L^{p,\alpha_1}(\R^3)}$ denotes the weighted $L^p$ norm defined by
\begin{equation*}
\|u\|_{L^{p,\alpha_1}(\R^3)}:=\left(\int_{\R^3}(1+|x|^{-\alpha_1})^{2-p}|u(x)|^p\mathrm{d}x\right)^{1/p},\quad \forall p\geq1.
\end{equation*}
It is easy to see that the (short time) dispersive estimate \eqref{dis:S} for the Schr\"odinger propagator $e^{-itH_{a,1}}$ is an analog of the time-decay estimate \eqref{dis:inverse} for $a\geq0$ (note that $H_{a,0}=-\Delta+a|x|^{-2}$). Analogous to \eqref{w-dis:inverse}, it should hold the following weighted decay estimate for $-\frac{1}{4}<a<0$
\begin{equation}\label{w-dis:S}
\|e^{-itH_{a,1}}\|_{L^{p,\alpha_1}(\mathbb{R}^3)\rightarrow L^{p',\alpha_1}(\mathbb{R}^3)}\lesssim
\frac{(1+|\sin t|^{\alpha_1})^{1-\frac{2}{p'}}}{|\sin t|^{3(\frac{1}{2}-\frac{1}{p'})}},\quad t\neq k\pi,k\in\mathbb{Z}.
\end{equation}
Since $\alpha_1>0$ when $-\frac{1}{4}<a<0$, the decay estimate \eqref{w-dis:inverse} is weaker than the usual one \eqref{dis:inverse} and so is \eqref{w-dis:S}. We cannot recover the dispersive estimate \eqref{dis:S} for the regime $-\frac{1}{4}<a<0$ since the weight function in \eqref{w-dis:inverse} seems to be indispensable (actually, the weighted estimate \eqref{w-dis:inverse} seems to be best possible, in view of \cite[Remark 1.12]{FFFP13}) and this is exactly the reason why we restrict to consider the case $a\geq0$ in the present paper. In addition, it is not known whether the weighted estimate \eqref{w-dis:inverse} (or \eqref{w-dis:S}) holds for the critical case $a=-\frac{1}{4}$. Nevertheless, these unsolved problems may be left as the main objects of further investigations.
\end{remark}

\section{heat kernel}

In this section, we construct a representation formula for the kernel of the heat flow $e^{-tH_{a,1}}$ and then prove the Gaussian boundedness. The study of heat kernels for Schr\"odinger operators with potentials on Euclidean spaces $\R^d$ or manifolds has its own interest.

\begin{proposition}[Heat kernel]\label{prop:H-ker}
Let $H_{a,1}$ be the operator \eqref{H-ab} and let $u$ be the unique solution to the Schr\"odinger equation
\begin{equation*}
\begin{cases}
\partial_tu(t,x)+H_{a,1}u(t,x)=0,\\
u(0,x)=f(x)\in L^2(\R^3),
\end{cases}
\end{equation*}
then the solution $u$ can be represented as
\begin{equation}\label{sol-rep:H}
u(t,x)=\left(e^{-tH_{a,1}}f\right)(x):=\int_{\R^3}K^H(x,y)f(y)\mathrm{d}y,
\end{equation}
where $K^H(x,y)$ corresponds to the kernel of the heat flow $e^{-tH_{a,1}}$.

Moreover, if we write $K_t^H(r_1,\vartheta_1;r_2,\vartheta_2)$ for the kernel $K^H(x,y)$ in polar coordinates $(r,\vartheta)$,
then we have
\begin{align}\label{ker:H}
K_t^H(r_1,\vartheta_1;r_2,\vartheta_2)=\frac{e^{-\frac{r_1^2+r_2^2}{4}\coth t}}{2\sqrt{r_1r_2}\sinh t}\sum_{k=0}^\infty I_{\beta_k}\Bigg(\frac{r_1r_2}{2\sinh t}\Bigg)Z^{(k)}_\vartheta(\vartheta'),
\end{align}
where $\beta_k$ is defined in \eqref{alpha-beta} and $Z^{(k)}_\vartheta(\vartheta')$ is the zonal function in \eqref{zonal:sphe}.
\end{proposition}
\begin{remark}
The integral at the right hand side of \eqref{sol-rep:H} should be understood in the sense of improper integrals as the Schr\"odinger counterpart, i.e.
\begin{equation*}
u(t,x)=\lim_{R\rightarrow+\infty}\int_{B_R(0)}K^H(x,y)f(y)\mathrm{d}y,\quad \text{where}\quad B_R(0):=\{y\in\R^3:|y|<R\}.
\end{equation*}
\end{remark}
\begin{proof}[Proof of Proposition \ref{prop:H-ker}]
The proof is similar to Proposition \ref{prop:S-ker} for the Schr\"odinger case. Since the spectrum of the operator $H_{a,1}$ consists of pure points, we have, by the spectral theorem,
\begin{equation}\label{heat:1}
e^{-tH_{a,1}}(x,y)=\sum_{m,k\in\mathbb{N}}e^{-t\lambda_{m,k}}e_{m,k}(x)\overline{e_{m,k}(y)}.
\end{equation}
Substituting the eigenvalues \eqref{eigen-v} and eigenfunctions \eqref{eigen-f} into \eqref{heat:1}, one gets
\begin{align}\label{heat:2}
K_t^H(r_1,\vartheta_1;r_2,\vartheta_2)
=(r_1r_2)^{-1/2}&e^{-\frac{r_1^2+r_2^2}{4}}\sum_{k=0}^\infty\frac{(r_1r_2)^{\beta_k}e^{-t(1+\beta_k)}}{2^{\beta_k}}\psi_k(\vartheta_1)\overline{\psi_k(\vartheta_2)}\nonumber\\
&\qquad\times\sum_{m=0}^\infty e^{-2mt}\frac{m!L_m^{\beta_k}\left(\frac{r_1^2}{2}\right)L_m^{\beta_k}\left(\frac{r_2^2}{2}\right)}{\Gamma(m+1+\beta_k)}.
\end{align}
By the formula \eqref{formula:sum}, we obtain
\begin{equation}\label{heat:3}
\begin{split}
K_t^H(r_1,\vartheta_1;r_2,\vartheta_2)=\frac{(r_1r_2)^{-1/2}e^{-\frac{r_1^2+r_2^2}{4\tan t}}}{2\sinh t}
\sum_{k=0}^\infty\psi_k(\vartheta_1)\overline{\psi_k(\vartheta_2)}I_{\beta_k}\Bigg(\frac{r_1r_2}{2\sinh t}\Bigg).
\end{split}
\end{equation}
Finally, in view of \eqref{psi-k} and \eqref{zonal:sphe}, the desired formula \eqref{ker:H} follows.
\end{proof}
In the following, we prove a Gaussian type upper bound for the heat kernel in Proposition \ref{prop:H-ker}.
\begin{proposition}\label{prop:H-Guass}
Let $H_{a,1}$ be the operator \eqref{H-ab}, then there exists some constant $C>0$ such that
\begin{equation}\label{dis:H}
|e^{-tH_{a,1}}(x,y)|\leq C|\sinh t|^{-\frac{3}{2}}e^{-\frac{|x-y|^2}{4\tanh t}},\quad \forall t>0.
\end{equation}
\end{proposition}
\begin{proof}
By using the function $K(\cdot,\cdot,\cdot)$ defined in Lemma \ref{lem:sum-ker}, the heat kernel \eqref{ker:H} can be rewritten as
\begin{equation*}
e^{-tH_{a,1}}(x,y)=e^{-\frac{|x|^2+|y|^2}{4\tanh t}}(2\sinh t)^{-\frac{3}{2}}K\left(\frac{\rho}{2\sinh t},\frac{x}{|x|},\frac{y}{|y|}\right),\quad t>0.
\end{equation*}
Then the desired bound \eqref{dis:H} follows by Lemma \ref{lem:sum-ker}.
\end{proof}
\begin{remark}
Note that for small time $t$ (for example $t\in(0,1)$), the bound \eqref{dis:H} can be regarded as the usual Gaussian upper bound $t^{-3/2}e^{-\frac{|x-y|^2}{4t}}$ for heat kernels due to the simple facts $\sinh t\sim t\sim\tanh t$ as $t\rightarrow0$. However, for large time $t\gg1$, one can see a exponential decay $e^{-\frac{3}{2}t}$ for the heat kernel, which is faster than the usual polynomial decay $t^{-3/2}$, due to the fact $\sinh t\sim e^t$ as $t\rightarrow+\infty$. The Gaussian bound in \eqref{dis:H} is known to play an essential role in proving the corresponding Bernstein type inequalities, which will be used to obtain the dispersive estimate for the half-wave propagator.
\end{remark}

\section{Bernstein inequalities, Besov spaces \& subordination formula}

In this section, we prove the Bernstein inequalities, define the Besov spaces and give the subordination formula adapted to the present setting.

\subsection{Bernstein inequalities}

In this subsection, we prove the $L^p$ Bernstein inequalities concerning the frequency truncations of the Schr\"odinger operator $H_{a,1}$ given by \eqref{H-ab}. In the following lemma, we prove the Bernstein inequalities for more general cases $1\leq p\leq q\leq\infty,s\in\R$, although only the case $(p,q,s)=(1,\infty,0)$ is used to prove the main theorem (Theorem \ref{thm:wave}).
\begin{lemma}[Bernstein inequalities]\label{lem:Bernstein}
Let $\psi\in C_c^\infty(\R)$ with $\psi\in[0,1]$ and $\mathrm{supp}\psi\subset[1/2,2]$. Then, for any $f\in L^p(\R^3),s\in\R$ and $j\in\mathbb{Z}$, it holds
\begin{equation}\label{eq:Bernstein}
\|\psi(2^{-j}\sqrt{H_{a,1}})H_{a,1}^sf\|_{L^q(\R^3)}\leq C2^{2sj+3j\left(\frac{1}{p}-\frac{1}{q}\right)}\|f\|_{L^p(\R^3)},\quad 1\leq p\leq q\leq\infty
\end{equation}
for some constant $C>0$.
\end{lemma}
\begin{proof}
The proof is rather standard as in the theory of function spaces associated to Schr\"odinger operators with potentials and we sketch it here for completeness.

\noindent From Proposition \ref{prop:H-Guass}, it is known that the heat kernel of the operator $H_{a,1}$ has a Gaussian upper bound, i.e.
\begin{equation*}
|e^{-tH_{a,1}}(x,y)|\leq Ct^{-\frac{3}{2}}e^{-\frac{|x-y|^2}{ct}},\quad\text{for some}\quad c,C>0.
\end{equation*}
By \cite[Proposition 5.1]{OZ06}, we conclude that the integral kernel of the truncated operator $\psi(2^{-j}\sqrt{H_{a,1}})$ satisfies
\begin{equation}\label{kernel:bound}
|\psi(2^{-j}\sqrt{H_{a,1}})(x,y)|\leq\frac{C_N2^{3j}}{(1+2^j|x-y|)^N},
\end{equation}
for all $N>0$, where $C_N>0$ is a constant independent of $j$. The bound can be extended by scaling to a slightly more general case, i.e.
\begin{equation*}
|H_{a,1}^s\psi(2^{-j}\sqrt{H_{a,1}})(x,y)|\leq\frac{C_N2^{2sj+3j}}{(1+2^j|x-y|)^N}.
\end{equation*}
Then the desired inequality \eqref{eq:Bernstein} follows by applying Young's inequality:
\begin{align*}
\|H_{a,1}^s\psi(2^{-j}\sqrt{H_{a,1}})f\|_{L^q(\R^3)}&=\left\|\int_{\R^3}H_{a,1}^s\psi(2^{-j}\sqrt{H_{a,1}})(x,y)f(y)\mathrm{d}y\right\|_{L^q(\R^3)}\\
&\leq C_N2^{2sj+3j}\left\|\int_{\R^3}\frac{|f(y)|\mathrm{d}y}{(1+2^j|x-y|)^N}\right\|_{L^q(\R^3)}\\
&\leq C_N2^{2sj+3j}\|(1+2^j|\cdot|)^{-N}\|_{L^r(\R^3)}\|f\|_{L^p(\R^3)}\\
&\leq C_N2^{2sj+3j\left(\frac{1}{p}-\frac{1}{q}\right)}\|f\|_{L^p(\R^3)},
\end{align*}
where $1+\frac{1}{q}=\frac{1}{r}+\frac{1}{p}$.
\end{proof}

\subsection{Besov spaces}

In this subsection, we define the Besov spaces in the present setting for general cases $1\leq p\leq q\leq\infty,s\in\R$, although only the case $(p,q,s)=(1,1,1)$ is needed in the main theorem (Theorem \ref{thm:wave}). Take a partition of unity
\begin{equation*}
1=\sum_{j\in\mathbb{Z}}\psi(2^{-j}\lambda):=\sum_{j\in\mathbb{Z}}\psi_j(\lambda),\quad \lambda\in(0,+\infty),
\end{equation*}
where $\psi\in C_c^\infty$ with $0\leq\psi\leq1$ and $\mathrm{supp}\psi\subset[1/2,2]$. For $1\leq p,q<\infty$ and $s\in\R$, the Besov space $\dot{\mathcal{B}}^{s,\psi}_{p,q}$ associated to the operator $H_{a,1}$ is defined by
\begin{equation}\label{space:Besov}
\dot{\mathcal{B}}^{s,\psi}_{p,q}(\R^3):=\left\{f\in L^2(\R^3):\|f\|^q_{\dot{\mathcal{B}}^{s,\psi}_{p,q}}=\sum_{j\in\mathbb{Z}}2^{jqs}\|\psi_j(\sqrt{H_{a,1}})f\|_{L^p(\R^3)}^q<+\infty\right\}.
\end{equation}
As the classical Besov norm, the distorted Besov norm in \eqref{space:Besov} is independent of the choice of a partition of unity in the following sense.
\begin{proposition}\label{prop:equiva}
Let $\varphi,\psi\in C_c^\infty(\R)$ be two different partitions of unity with supports in $[1/2,2]$.
Then, for $1\leq p,q<\infty$ and $s\in\R$, we have
\begin{equation}\label{Bes:equiv}
\|f\|_{\dot{\mathcal{B}}^{s,\psi}_{p,q}(\R^3)}\sim\|f\|_{\dot{\mathcal{B}}^{s,\varphi}_{p,q}(\R^3)}.
\end{equation}
\end{proposition}
\begin{proof}
To prove the equivalence \eqref{Bes:equiv}, we claim that for any $m\in\mathbb{N}$, there exists a constant $C_m>0$ such that
\begin{equation}\label{eq:j-k}
\|\psi_j(\sqrt{H_{a,1}})\varphi_k(\sqrt{H_{a,1}})f\|_{L^p(\R^3)}\leq C_m2^{-2m|j-k|}\|f\|_{L^p(\R^3)},\quad \forall j,k\in\mathbb{Z}.
\end{equation}
We only consider the case $j\geq k$ in \eqref{eq:j-k} and the other case can be treated similarly. For $m\in\mathbb{N}$, we write
\begin{equation*}
\psi_j(\sqrt{H_{a,1}})\varphi_k(\sqrt{H_{a,1}})f=H_{a,1}^{-m}\psi_j(\sqrt{H_{a,1}})H_{a,1}^m\varphi_k(\sqrt{H_{a,1}})f
\end{equation*}
which, together with Bernstein inequalities \eqref{eq:Bernstein}, implies that
\begin{equation*}
\begin{split}
\|\psi_j(\sqrt{H_{a,1}})\varphi_k(\sqrt{H_{a,1}})f\|_{L^p(\R^3)}
&\leq C_m2^{-2jm}\|H_{a,1}^m\varphi_k(\sqrt{H_{a,1}})f\|_{L^p(\R^3)}\\
&\leq C_m2^{-2(j-k)m}\|f\|_{L^p(\R^3)}.
\end{split}
\end{equation*}
Then the claimed estimate \eqref{eq:j-k} follows.
To prove \eqref{Bes:equiv}, since the roles of $\varphi$ and $\psi$ are equivalent, it is sufficient to show that
\begin{equation*}
\|f\|_{\dot{\mathcal{B}}^{s,\psi}_{p,q}(\R^3)}\leq C\|f\|_{\dot{\mathcal{B}}^{s,\varphi}_{p,q}(\R^3)},\quad f\in\dot{\mathcal{B}}^{s,\varphi}_{p,q}(\R^3).
\end{equation*}
Since $\varphi$ is a partition of unity, it holds
\begin{equation*}
\psi_j(\sqrt{H_{a,1}})f=\sum_{k\in\mathbb{Z}}\psi_j(\sqrt{H_{a,1}})\varphi_k(\sqrt{H_{a,1}})f,\quad j\in\mathbb{Z}.
\end{equation*}
On the other hand, one has $\varphi_k=(\varphi_{k-1}+\varphi_k+\varphi_{k+1})\varphi_k$ in view of the support property of $\varphi$. Hence, we have
\begin{equation*}
\psi_j(\sqrt{H_{a,1}})f=\sum_{k\in\mathbb{Z}}\psi_j(\sqrt{H_{a,1}})(\varphi_{k-1}(\sqrt{H_{a,1}})+\varphi_k(\sqrt{H_{a,1}})
+\varphi_{k+1}(\sqrt{H_{a,1}}))\varphi_k(\sqrt{H_{a,1}})f,
\end{equation*}
which implies
\begin{equation*}
\|\psi_j(\sqrt{H_{a,1}})f\|_{L^p(\R^3)}\leq\sum_{k\in\mathbb{Z}}\|\psi_j(\sqrt{H_{a,1}})
(\varphi_{k-1}(\sqrt{H_{a,1}})+\varphi_k(\sqrt{H_{a,1}})+\varphi_{k+1}(\sqrt{H_{a,1}}))\varphi_k(\sqrt{H_{a,1}})f\|_{L^p(\R^3)}.
\end{equation*}
We now apply \eqref{eq:j-k} to deduce that for $m>|s|/2$
\begin{equation}\label{eq:dyadic}
\|\psi_j(\sqrt{H_{a,1}})f\|_{L^p(\R^3)}\leq\sum_{k\in\mathbb{Z}}2^{-2|j-k|m}\|\varphi_k(\sqrt{H_{a,1}})f\|_{L^p(\R^3)}.
\end{equation}
Inserting \eqref{eq:dyadic} into the expression of the Besov norm in \eqref{space:Besov} gives
\begin{align*}
\|f\|^q_{\dot{\mathcal{B}}^{s,\psi}_{p,q}(\R^3)}&=\sum_{j\in\mathbb{Z}}2^{jsq}\|\psi_j(\sqrt{H_{a,1}})f\|_{L^p(\R^3)}^q\\
&\leq C\sum_{j\in\mathbb{Z}}2^{jsq}\left(\sum_{k\in\mathbb{Z}}2^{-2|j-k|m}\|\varphi_k(\sqrt{H_{a,1}})f\|_{L^p(\R^3)}\right)^q\\
&\leq C\sum_{j\in\mathbb{Z}}\left(\sum_{k\in\mathbb{Z}}2^{-2|j-k|m+js}\|\varphi_k(\sqrt{H_{a,1}})f\|_{L^p(\R^3)}\right)^q
\end{align*}
which, together with Young's inequality, implies
\begin{equation*}
\|f\|_{\dot{\mathcal{B}}^{s,\psi}_{p,q}(\R^3)}\leq C\left(\sum_{k\in\mathbb{Z}}2^{sqk}\|\varphi_k(\sqrt{H_{a,1}})f\|_{L^p(\R^3)}^q\right)^{1/q}
=\|f\|_{\dot{\mathcal{B}}^{s,\varphi}_{p,q}(\R^3)}
\end{equation*}
as desired.
\end{proof}
\noindent In view of Proposition \ref{prop:equiva}, for $1\leq p,q<\infty$ and $s\in\R$, we may briefly write $\|\cdot\|_{\dot{\mathcal{B}}^s_{p,q}}$ for $\|\cdot\|_{\dot{\mathcal{B}}^{s,\psi}_{p,q}(\R^3)}$ if there is no confusion, where $\psi\in C_c^\infty(\R)$ is a partition of unity with $0\leq\psi\leq1$ and $\mathrm{supp}\psi\subset[1/2,2]$.

\subsection{Subordination formula}

In this subsection, we give a simplified version of the classical subordination formula introduced by M\"uller and Seeger \cite{MS08}.
\begin{proposition}[Subordination formula]\label{prop:subord}
Let $\eta\in C_c^\infty(\R)$ with $\mathrm{supp}\eta\subset[1/2,2]$.
Then there exist a symbol function $a\in C^\infty(\R\times\R)$ with $\mathrm{supp}a(\cdot,\sigma)\subset[1/16,4]$ for every $\sigma\in\R$ fulfilling
\begin{equation*}
|\partial_s^j\partial_\sigma^ka(s,\sigma)|\leq C_{j,k}(1+|s|)^{-k},\quad \forall j,k\geq0
\end{equation*}
and a Schwartz function $\rho(s,\sigma)\in\mathcal{S}(\R\times\R)$ such that
\begin{equation}\label{formula}
\eta(\sqrt{x}/2^j)e^{it\sqrt{x}}=\rho(t\sqrt{x}/2^j,2^jt)+\eta(\sqrt{x}/2^j)2^{j/2}\sqrt{t}\int_0^\infty a(s,2^jt)e^{i2^{j-2}t/s}e^{istx/2^j}\mathrm{d}s,
\end{equation}
for every $x\geq0$ and $t\geq2^{-j}$ with $j\in\mathbb{Z}$.
\end{proposition}
\noindent The proof of Proposition \ref{prop:subord} can be found in the paper of M\"uller and Seeger \cite{MS08} (see also D'Ancona, Pierfelice and Ricci \cite{DPR10} in Appendix); we omit the details here.

\section{proof of Theorem \ref{thm:wave}}

With all the required ingredients in hand, we are now in the right position to prove the dispersive estimate for the wave equation (Theorem \ref{thm:wave}). For the convenience of the reader, we list these ingredients again: dispersive estimate for Schr\"odinger flow \eqref{dis:S} in Proposition \ref{prop:S-dis} of Section 3, Gaussian bound for heat kernel \eqref{dis:H} in Proposition \ref{prop:H-Guass} of Section 4, Bernstein inequalities \eqref{eq:Bernstein} in Proposition \ref{lem:Bernstein} and subordination formula \eqref{formula} in Proposition \ref{prop:subord} of Section 5.

To prove Theorem \ref{thm:wave}, a crucial ingredient is the following frequency-localized decay estimate for the half-wave propagator $e^{it\sqrt{H_{a,1}}}$ since the desired dispersive estimate \eqref{dis:wave} then follows from the definition of the distorted Besov space $\dot{\mathcal{B}}^1_{1,1}(\R^3)$ by summing all the frequencies.
\begin{lemma}\label{lem:disper}
Let $0<T<\pi$ and $\psi\in C_c^\infty(\R)$ with $0\leq\psi(\lambda)\leq1$ and $\mathrm{supp}\psi\subset[1/2,2]$. Then, for all $j\in\mathbb{Z},|t|<T$ and $f\in L^1(\R^3)$, there exists a constant $C>0$ independent of $t$ such that
\begin{equation}\label{dis:half-wave}
\|\psi_j(\sqrt{H_{a,1}})e^{it\sqrt{H_{a,1}}}f\|_{L^\infty(\R^3)}\leq C2^{-3j}\langle2^jt\rangle^{-1}\|f\|_{L^1(\R^3)},
\end{equation}
where $\psi_j(\lambda):=\psi(2^{-j}\lambda)$ and $\langle\cdot\rangle:=(1+|\cdot|^2)^{1/2}$.
\end{lemma}
\begin{proof}
It suffices to consider the case $t>0$ since the case $t<0$ can be treated similarly. We follow the proof of the main result of \cite{DPR10}.
Note that $\psi_j=\psi_j(\psi_{j-1}+\psi_j+\psi_{j+1})$, then we have for $t<2^{-j}$
\begin{equation*}
\|\psi_j(\sqrt{H_{a,1}})e^{it\sqrt{H_{a,1}}}f\|_{L^\infty(\R^3)}
\leq\sum_{k=0}^2\|\psi_j(\sqrt{H_{a,1}})e^{it\sqrt{H_{a,1}}}\psi_{j+k-1}(\sqrt{H_{a,1}})f\|_{L^\infty(\R^3)}.
\end{equation*}
By Lemma \ref{lem:Bernstein}, we obtain
\begin{align*}
\|\psi_j(\sqrt{H_{a,1}})e^{it\sqrt{H_{a,1}}}f\|_{L^\infty(\R^3)}
&\leq C\sum_{k=0}^2\|\psi_j(\sqrt{H_{a,1}})\|_{L^2(\R^3)\rightarrow L^\infty(\R^3)}\|e^{it\sqrt{H_{a,1}}}\|_{L^2(\R^3)\rightarrow L^2(\R^3)}\\
&\qquad\qquad \times\|\psi_{j+k-1}(\sqrt{H_{a,1}})\|_{L^1(\R^3)\rightarrow L^2(\R^3)}\|f\|_{L^1(\R^3)}\\
&\leq C2^{3j}\|f\|_{L^1(\R^3)}\sim\frac{2^{3j}}{1+2^j|t|}\|f\|_{L^1(\R^3)}.
\end{align*}
For $t\geq2^{-j}$, we apply Proposition \ref{prop:subord} to write
\begin{align*}
\psi_j(\sqrt{H_{a,1}})e^{it\sqrt{H_{a,1}}}f&=\rho(2^{-j}t\sqrt{H_{a,1}},2^jt)f\\
&\quad+\psi_j(\sqrt{H_{a,1}})2^\frac{j}{2}\sqrt{t}\int_0^\infty a(s,2^jt)e^{i2^{j-2}t/s}e^{i2^{-j}stH_{a,1}}f\mathrm{d}s\\
&:=S_1+S_2.
\end{align*}
For the first term $S_1$, by applying Lemma \ref{lem:Bernstein} and noting the fact $\rho\in\mathcal{S}(\R\times\R)$, we have
\begin{align*}
|S_1|&\leq C\|\tilde{\psi}_j(\sqrt{H_{a,1}})\|_{L^2(\R^3)\rightarrow L^\infty(\R^3)}\|\rho(2^{-j}t\sqrt{H_{a,1}},2^jt)\|_{L^2(\R^3)\rightarrow L^2(\R^3)}\\
&\qquad\quad \times\|\tilde{\psi}_j(\sqrt{H_{a,1}})\|_{L^1(\R^3)\rightarrow L^2(\R^3)}\|f\|_{L^1(\R^3)}\\
&\leq C 2^{3j}\|\rho(2^{-j}t\sqrt{x},2^jt)\|_{L_x^\infty(\R)}\|f\|_{L^1(\R^3)}\\
&\leq C 2^{3j}(2^jt)^{-1}\|f\|_{L^1(\R^3)}\\
&\leq C 2^{2j}t^{-1}\|f\|_{L^1(\R^3)}\\
&\sim\frac{2^{3j}}{1+2^j|t|}\|f\|_{L^1(\R^3)},
\end{align*}
where $\tilde{\psi}_j=\psi_{j-1}+\psi_j+\psi_{j+1}$.

\noindent For the second term $S_2$, by \eqref{dis:S} (dispersive estimate for $e^{-itH_{a,1}}$) and Lemma \ref{lem:Bernstein}, we obtain
\begin{align*}
|S_2|&\leq C2^\frac{j}{2}\sqrt{t}\int_0^\infty|a(s,2^jt)|(2^{-j}st)^{-3/2}\|f\|_{L^1(\R^3)}\mathrm{d}s\\
&\leq C2^{2j}t^{-1}\|f\|_{L^1(\R^3)}\sim\frac{2^{3j}}{1+2^j|t|}\|f\|_{L^1(\R^3)}.
\end{align*}
Therefore, the proof of \eqref{dis:half-wave} is completed.
\end{proof}
\begin{remark}
Theorem \ref{thm:wave} can be derived from Lemma \ref{lem:disper} by the definition of the distorted Besov space $\dot{\mathcal{B}}^1_{1,1}$ in \eqref{Besov}. Indeed, let $\tilde{\psi}_j=\tilde{\psi}(2^{-j}\sqrt{H_{a,1}})$ and $\tilde{\psi}(\lambda)=\lambda^{-1}\psi(\lambda)$ with $\psi$ as in Lemma \ref{lem:disper}.
Note that from \eqref{dis:half-wave}, we have
\begin{equation*}
\|e^{it\sqrt{H_{a,1}}}\tilde{\psi}_jf\|_{L^\infty(\R^3)}\leq C2^{2j}|t|^{-1}\|f\|_{L^1(\R^3)},
\end{equation*}
which implies
\begin{align*}
\left\|\frac{e^{it\sqrt{H_{a,1}}}}{\sqrt{H_{a,1}}}f\right\|_{L^\infty(\R^3)}
&\leq\sum_{j\in\mathbb{Z}}\left\|\frac{e^{it\sqrt{H_{a,1}}}}{\sqrt{H_{a,1}}}\psi_j(\sqrt{H_{a,1}})f\right\|_{L^\infty(\R^3)}\\
&\leq C|t|^{-1}\sum_{j\in\mathbb{Z}}2^{2j}\|H^{-1/2}_{a,1}\psi_j(\sqrt{H_{a,1}})f\|_{L^1(\R^3)}\\
&=C|t|^{-1}\sum_{j\in\mathbb{Z}}2^j\|\tilde{\psi}_j(\sqrt{H_{a,1}})f\|_{L^1(\R^3)}\\
&=C|t|^{-1}\|f\|_{\dot{\mathcal{B}}^{1,\tilde{\psi}}_{1,1}(\R^3)}\\
&\sim|t|^{-1}\|f\|_{\dot{\mathcal{B}}^1_{1,1}(\R^3)}.
\end{align*}
\end{remark}
\begin{remark}
In view of Remark \ref{rem:negative}, we point out that the argument of obtaining the dispersive estimate for the wave equation in Theorem \ref{thm:wave} fails if $-\frac{1}{4}\leq a<0$. A weighted version of the dispersive estimate \eqref{dis:wave} should be valid for this case $-\frac{1}{4}\leq a<0$, but this is open question.
Nevertheless, motivated by \cite[Remark 1.2]{FZZ22}, it is interesting to conjecture that the corresponding Strichartz estimates as in Theorem \ref{thm:str} should be true for the case $-\frac{1}{4}<a<0$ or even the critical case $a=-\frac{1}{4}$ as Mizutani \cite{Miz17} for the Schr\"odinger equation; these problems may be the main objects of future investigations.
\end{remark}

\section{proof of Theorem \ref{thm:str}}

In this section, we prove Strichartz estimates in Theorem \ref{thm:str}. It is sufficient to show that for admissible $q,r$ in \eqref{adm:scaling} and $s$ in \eqref{cd:gap}, it holds
\begin{equation}\label{est:strichartz}
\|e^{it\sqrt{H_{a,1}}}f\|_{L_t^q(I;L_x^r(\R^3))}\leq C\|f\|_{\dot{H}_{osc}^s(\R^3)},
\end{equation}
where $C>0$ is a constant independent of $f\in\dot{H}_{osc}^s(\R^3)$ and $I=[a,b]\subset(0,\pi)$.
\begin{proposition}\label{prop:str}
Let $\psi\in C_c^\infty(\R)$ and $\psi_j(\lambda)=\psi(2^{-j}\lambda)$ with $0\leq\psi(\lambda)\leq1$ and $\mathrm{supp}\psi\subset[1/2,2]$. Then, for all $j\in\mathbb{Z}$ and $f\in L^2(\R^3)$, there exists some constant $C>0$ such that
\begin{equation}\label{str:half-wave}
\|\psi_j(\sqrt{H_{a,1}})e^{it\sqrt{H_{a,1}}}f\|_{L_t^q(I;L_x^r(\R^3))}\leq C2^{2sj}\|f\|_{L^2(\R^3)},
\end{equation}
with $I=[a,b]\subset(0,\pi)$ and $(q,r),s$ in \eqref{adm:scaling}, \eqref{cd:gap}, respectively.
\end{proposition}
\begin{proof}
First, we prove that
\begin{align}\label{eq:h-wave}
&\left\|\left(\sum_{j\in\mathbb{Z}}|\psi_j(\sqrt{H_{a,1}})e^{it\sqrt{H_{a,1}}}f|^2\right)^{1/2}\right\|_{L_t^q(I;L_x^r(\R^3))}\nonumber\\
&\leq C\left\|\left(\sum_{j\in\mathbb{Z}}|H_{a,1}^{\frac{s}{2}}\psi_j(\sqrt{H_{a,1}})e^{it\sqrt{H_{a,1}}}f|^2\right)^{1/2}\right\|_{L^2(\R^3)}.
\end{align}
Indeed, by Lemma \ref{lem:disper}, we obtain
\begin{align*}
\|\psi_j(\sqrt{H_{a,1}})e^{it\sqrt{H_{a,1}}}f\|_{L^\infty(\R^3)}
&=\|e^{it\sqrt{H_{a,1}}}\tilde{\psi}_j(\sqrt{H_{a,1}})\psi_j(\sqrt{H_{a,1}})f\|_{L^\infty(\R^3)}\\
&\leq C\frac{2^{3j}}{1+2^j|t|}\|\psi_j(\sqrt{H_{a,1}})f\|_{L^1(\R^3)},
\end{align*}
where $\tilde{\psi}_j=\psi_{j-1}+\psi_j+\psi_{j+1}$.
Moreover, by the unitary property of the propagator $e^{it\sqrt{H_{a,1}}}$
\begin{equation*}
\|e^{it\sqrt{H_{a,1}}}\psi_j(\sqrt{H_{a,1}})f\|_{L^2(\R^3)}=\|\psi_j(\sqrt{H_{a,1}})f\|_{L^2(\R^3)},
\end{equation*}
we can conclude, by following the argument of Keel-Tao \cite[Corollary 1.3]{KT98}, that
\begin{equation*}
\|e^{it\sqrt{H_{a,1}}}\psi_j(\sqrt{H_{a,1}})f\|_{L_t^q(I;L_x^r(\R^3))}\leq C\|H_{a,1}^{\frac{s}{2}}\psi_j(\sqrt{H_{a,1}})f\|_{L^2(\R^3)}.
\end{equation*}
Thus, the desired estimate \eqref{eq:h-wave} follows by Minkowski's inequality since $q,r\geq2$.

\noindent Next, by the almost orthogonality
\begin{equation}\label{eq:LP}
\left\|\left(\sum_{j\in\mathbb{Z}}|\psi_j(\sqrt{H_{a,1}})f|^2\right)^{1/2}\right\|_{L^r(\R^3)}\sim\|f\|_{L^r(\R^3)},\quad \forall1<r<\infty,
\end{equation}
we obtain
\begin{equation*}
\|e^{it\sqrt{H_{a,1}}}f\|_{L_t^q(I;L_x^r(\R^3))}\leq C\|H_{a,1}^{\frac{s}{2}}f\|_{L^2(\R^3)}.
\end{equation*}
Therefore, we have
\begin{equation*}
\left\|\cos(t\sqrt{H_{a,1}})f+\frac{\sin(t\sqrt{H_{a},1})}{\sqrt{H_{a,1}}}g\right\|_{L_t^q(I;L_x^r(\R^3))}\leq C\|f\|_{\dot{H}_{osc}^s(\R^3)}+\|g\|_{\dot{H}_{osc}^{s-1}(\R^3)}
\end{equation*}
as desired.
\end{proof}

%{\bf Acknowledgements:} The author is currently a graduate student in Beijing Institute of Technology and he would like to thank his advisor professor Junyong Zhang for introducing the topic of this paper. The author would like to express special appreciations to the anonymous reviewers for the insightful comments and valuable suggestions on improving the manuscript. The author claims no funds supported during the preparation of the present article.

\newpage

\vspace{0.8cm}
\begin{center}

\end{center}

\end{document}